\documentclass[11pt]{article}
\usepackage[english]{babel}
\usepackage{amsmath,amsthm}
\usepackage{amsfonts}
\usepackage{cite}
\usepackage{bbm}
\usepackage{color}
\usepackage{mathrsfs}

\newtheorem{Theorem}{Theorem}[section]
\newtheorem{Corollary}[Theorem]{Corollary}
\newtheorem{Lemma}[Theorem]{Lemma}

\theoremstyle{definition}
\newtheorem{Definition}[Theorem]{Definition}
\newtheorem{assumption}[Theorem]{Assumption}
\theoremstyle{remark}
\newtheorem{Remark}[Theorem]{Remark}
\numberwithin{equation}{section}
\begin{document}
\def\Pro{{\mathbb{P}}}
\def\E{{\mathbb{E}}}
\def\e{{\varepsilon}}
\def\ds{\displaystyle}
\def\sspan{{\textnormal{span}}}
\def\nat{{\mathbb{N}}}
\def\Tr{{\text{Tr}}}

\title{Rare event simulation via importance sampling for linear SPDE's}%
\author{Michael Salins and Konstantinos Spiliopoulos\footnote{Department of Mathematics \& Statistics, Boston University, 111 Cummington Street, Boston MA 02215, E-mail: msalins@bu.edu, kspiliop@math.bu.edu. K.S. was partially supported  by the National Science Foundation CAREER award DMS 1550918.}}%
\maketitle
\begin{abstract}
The goal of this paper is to develop provably efficient importance sampling Monte Carlo methods for the estimation of rare events within the class of linear stochastic partial differential equations (SPDEs). We find that if a spectral gap of appropriate size exists, then one can identify a lower dimensional manifold where the rare event takes place. This allows one to build importance sampling changes of measures that perform provably well even pre-asymptotically (i.e. for small but non-zero size of the noise) without degrading in performance due to infinite dimensionality or due to long simulation time horizons. Simulation studies supplement and illustrate the theoretical results.
\end{abstract}

\section{Introduction}
Consider the linear stochastic parabolic equation for $(t,\xi) \in [0,+\infty)\times \mathcal{O}$, where $\mathcal{O} \subset \mathbb{R}^d$ is a bounded, regular domain,
\begin{equation} \label{eq:intro-X}
  \begin{cases}
  \frac{\partial{X^\e}}{\partial t}(t,\xi) = A X^\e(t,\xi)  + \sqrt{\e} B \frac{\partial w}{\partial t}(t,\xi)\\
  X^\e(t,\xi) = 0, \ \xi \in \partial \mathcal{O}, \ \ \ \ X^\e(0,\xi) = x(\xi).
  \end{cases}
\end{equation}
The object $w(t)$ is a cylindrical Wiener process and $Q=BB^{\star}$ is the covariance operator of the noise. Its formal time derivative is a noise that is white in time and $Q$ correlated in space. Let $H =L^2(\mathcal{O})$ and define $A$ to be an unbounded linear operator that is self-adjoint such that there exists a sequence of eigenvalues $\{\alpha_{k}\}_{k\geq 1}$  and a sequence of eigenfunctions $\{e_{k}(\xi)\}_{k\geq 1}$ (that form a complete orthonormal basis of $H$) with the property that
\begin{equation}
  A e_k = -\alpha_k e_k, \ \ \ \ 0< \alpha_1 \leq \alpha_2 \leq \alpha_3 \leq ...
\end{equation}

In this way, $A$ can be an unbounded operator and it is the generator of the analytic semigroup $e^{tA}$, which has the property that
\[e^{tA} e_k = e^{-\alpha_k t} e_k.\]
A cylindrical Wiener process can be defined as the formal sum
\[w(t,\xi) = \sum_{k=1}^\infty  e_k(\xi) \beta_k(t)\]
where $\{\beta_k\}$ is a family of i.i.d. Brownian motions on some probability space $(\Omega, \mathcal{F},\Pro)$. $B:H \to H$ is a bounded linear operator which is positive definite and diagonalized by the eigenbasis $\{e_k\}$. In particular there is a collection of positive eigenvalues such that
\[B e_k = \lambda_k e_k.\]

\begin{assumption} \label{assum:eigens}
  The eigenvalues $\alpha_k$ and $\lambda_k$ satisfy
  \begin{equation}
    \sum_{k=1}^\infty \frac{\lambda_k^2}{\alpha_k} < +\infty.
  \end{equation}
\end{assumption}
Notice that Assumption \ref{assum:eigens} is equivalent to assuming $\left|A^{-1/2}B\right|_{2}<\infty$ where
 \[|A^{-1/2}B|_2^2 = \sum_{k=1}^\infty \left|A^{-1/2}B e_k\right|_H^2\]
is the Hilbert-Schmidt norm.
This assumption guarantees that a $L^2(\mathcal{O})$-valued solution to \eqref{eq:intro-X} exists. In the case that $A$ is the Laplace operator and $\mathcal{O}$ is sufficiently regular, $\alpha_k \sim k^{2/d}$, where $d$ is the spatial dimension. This means that in the case of spatial dimension $d\geq 2$, the eigenvalues of $B$ need to decay to zero. In the case $d=1$, $B$ can be the identity (white-noise case) and $B$ could even be unbounded.
Using this notation we can suppress the spatial variable by considering $X^\e$ as a stochastic process on the Hilbert space $H$ that solves the equation
\begin{equation}
  \begin{cases}
    dX^\e(t) = A X^\e(t) dt + \sqrt{\e} Bdw(t)\\
    X^\e(0)=x,
  \end{cases}
\end{equation}
the mild solution to \eqref{eq:intro-X} is given by
\begin{equation}
  X^\e(t) = e^{tA}x  + \sqrt{\e}\int_0^t e^{(t-s)A}Bdw(s)
\end{equation}
and we are interested in calculating the probability of exit from certain subsets of $H$ via Monte Carlo. Let $D \subset H$ and define $\tau^\e_x$ to be the exit time
\begin{equation}
  \tau^\e_x = \inf\{t>0: X^\e_x(t) \not \in D \}
\end{equation}
where $X^\e_x$ denotes the solution to $X^\e$ with initial condition $x$.

In this paper we focus on the case where $D=\{x\in H: |x|_H\leq L\}$ is the $L^{2}(\mathcal{O};\mathbb{R})$ ball with radius $L\in(0,\infty)$.  We are interested in developing provably efficient importance sampling schemes to calculate quantities such as
\begin{equation} \label{eq:intro-theta}
  \theta^\e(x,T) = \Pro \left(  \tau^\e_x \leq T  \right).
\end{equation}

$\theta^\e(x,T)$ is the distribution function of the exit time $\tau^\e_x$ of $X^{\e}(t)$ from the $L^{2}(\mathcal{O};\mathbb{R})$ ball with radius $L\in(0,\infty)$. Since, $L>0$ and $O$ is an attractor for the $\e=0$ noiseless dynamical system, we are dealing with a rare event. Estimation of rare events such as (\ref{eq:intro-theta}) could be of interest in several application domains, ranging from material science, where large $L^2$ norms could be events that one would like to avoid, to overflow of power grids where $O$ represents the steady state of the system and departures from it are unwanted, or to quantifying transitions between metastable states in nonlinear models.

As $\e$ gets smaller the event becomes rarer, i.e. $\theta^\e(x,T)$ becomes smaller, which makes estimation of such events difficult. Large deviations theory, see \cite{F-W-book,DaP-Z}, deals with approximation of quantities such as $\e\log(\theta^\e(x,T))$ in the small noise regime, i.e., when $\e\downarrow 0$. However, it ignores the effect of prefactors, which can be significant. In particular, even in dimension one, it has been established in \cite{dsz-2014} that in metastable cases one may have to go to very small values of $\e$ before one starts observing good behavior (in the sense of variance reduction of the estimators). The reason for this behavior is that other parameters of the problem such as the time horizon $T$ are important in the prelimit and compete with $1/\e$. For reasons like this accelerated Monte Carlo methods, like importance sampling, that also account for the behavior in the prelimit (i.e., before $\e\rightarrow 0$) become important.

We want to estimate quantities like $\theta^\e(x,T)$ via importance sampling. As it will be described in Section \ref{S:IS}, importance sampling is a variance reduction technique in Monte-Carlo simulation whose goal is to minimize the variance of the estimator via appropriate changes of measure. In this paper we focus on the linear case. The problem is challenging even then, as the curse of dimensionality kicks in rather quickly if the change of measure is not done correctly. This is an issue additional to the issues related to effect of other parameters of the problem in the prelimit, such as $T$, see also \cite{dsz-2014} for a related discussion in the finite dimensional case.

The main novelty of this paper lies in the identification of a condition on the spectral gap under which we can prove that the rare event takes place in a lower dimensional manifold. In particular, as Theorem \ref{thm:min-traj-e1} indicates, if $\lambda_1\geq \lambda_k$ and $3\alpha_1 < \alpha_k$ for all $k\geq 2$, then the rare event takes place in the $e_{1}$ direction. Even if $\lambda_1$ is not the maximal eigenvalue of $B$, then Theorem \ref{thm:min-traj-long-time} guarantees that if $2\frac{\alpha_1}{\lambda_1^2} < \frac{\alpha_k}{\lambda_k^2}$, then for large enough time horizons the rare event occurs in the $e_1$ direction. We also remark here that for metastability problems one is typically interested in long time horizons. Forcing the importance sampling trajectories in the $e_{1}$ direction results in them exiting along the $e_{1}$ direction with high probability, see Theorem \ref{thm:conv-of-endpoint}, and the relative errors of such an importance sampling scheme do not degrade as the dimension of the Galerkin approximation gets large, see Theorems \ref{T:Scheme1} and \ref{T:Scheme2}. To accomplish this we essentially project the desired change of measure down to the $e_{1}$ direction. In addition, the relative errors of the Monte Carlo simulations for the suggested changes of measure also do not degrade as the time horizon $T$ gets large. This latter behavior is parallel to what was achieved in the one-dimensional setting of \cite{dsz-2014}.

 To the best of our knowledge the current work is the first one to address provably efficient importance sampling schemes for SPDEs that perform well both in the limit as the noise goes to zero, but in the prelimit as well and do not degrade in performance due to increased dimension. The construction of the proposed  change of measures is based on the subsolution method introduced in \cite{DupuisWang2}.

The rest of the paper is organized as follows. In Section \ref{S:IS} we recall the importance sampling technique appropriately tailored to the infinite dimensional setting and we mention issues that come up due to the fact that we are dealing with infinite dimensions. In Section \ref{S:CalculusOfVariationsLinearProblem} we solve the corresponding deterministic variational problem and we provide conditions on the time horizon and eigenvalues that guarantee that the rare event likely takes place in the $e_{1}$ direction. In Section \ref{S:WeakCompactness} we are studying the $\e-$dependent problem and prove that if only the $e_{1}$ direction is being forced and the spectral gap mentioned in Section \ref{S:CalculusOfVariationsLinearProblem} holds, then the exit from the $L^{2}$ ball happens in the close neighborhood of  the $e_{1}$ manifold for $\e$ small enough. This allows to prove a non-asymptotic upper bound of the second moment of the importance sampling estimator in Section \ref{S:IS_schemes}.

 In Section \ref{S:IS_schemes} we also propose two specific changes of measure with provably good bounds in performance that do not degrade with increasing dimension, nor with increasing time horizon $T$. The first scheme, Scheme 1, is motivated by the one-dimensional construction of \cite{dsz-2014}, whereas the second proposed scheme, Scheme 2, is on the one hand simpler to apply and to analyze and computationally faster, but on the other hand performs slightly worse than Scheme 1. In Section \ref{S:Numerics} we present extensive simulation studies that demonstrate the theory developed in this paper. Section \ref{S:Conclusions} is on conclusions and future work where we also discuss the issues and challenges that come up in the direct application of the results of this paper in the nonlinear case.  An Appendix follows with the proof of the pre-asymptotic bound for the  performance of Scheme 2, and with an estimate with detailed dependence on $\e,T$ and $\alpha_{k}$ for the Galerkin approximation.

\section{Importance sampling in infinite dimensions and issues that come up}\label{S:IS}
In this section we recall what importance sampling is in the context of infinite dimensional models and discuss some of the issues that come up and are unique to the infinite dimensional setting. Assume that $w(t)$ is adapted to a filtration, $\mathcal{F}_t$.

For any $\mathcal{F}_t$-adapted $u^\e \in L^2([0,T];H)$, we can define the change of measure
\begin{equation}
  \frac{d\bar{\Pro}^\e}{d\Pro} = \exp \left(\frac{1}{\sqrt{\e}} \int_0^T \left<u^\e(s), dw(s) \right>_H - \frac{1}{2\e} \int_0^T |u^{\e}(s)|_H^2 ds \right).
\end{equation}
By Girsanov's formula in infinite dimensions (see, for example, \cite{DaP-Z})
\begin{equation}
  \bar{w}(t) = w(t) - \frac{1}{\sqrt{\e}} \int_0^t u^{\e}(s) ds
\end{equation}
is a cylindrical Wiener process in $\bar{\Pro}^\e$ and
\[dX^\e(t) = (AX^\e(t)  + Bu^\e(t))dt + \sqrt{\e}Bd\bar{w}(t)\]
Because of this, we can estimate $\theta^\e$ given in \eqref{eq:intro-theta} by using the unbiased estimator
\begin{equation}
 \Gamma^\e(0,x,u^\e) = 1_{\{\tau^\e_x \leq T \}} \frac{d \Pro}{d\bar{\Pro}^\e},
\end{equation}
where the first two entries in $\Gamma^{\e}$ denote that the initial condition is $X^{\e}(0)=x$. Our goal is to find a sequence of controls $u^\e$ so that the Monte Carlo estimator has small relative error. To this end, we try to minimize
\begin{equation} \label{eq:variance-est}
  \mathcal{Q}^{\e}(0,x,u^\e) = \bar{\E}^\e \left[ 1_{\{\tau^\e_x \leq T \}} \left( \frac{d \Pro}{d \bar{\Pro}^\e} \right)^2 \right].
\end{equation}
The explicit representation for $d\Pro/d\bar{\Pro}^\e$ is
\[\frac{d\Pro}{d\bar{\Pro}^\e} = \exp \left(-\frac{1}{\sqrt{\e}} \int_0^T \left< u^\e(s), d\bar{w}(s)\right>_H - \frac{1}{2\e} \int_0^T |u^{\e}(s)|_H^2 ds \right).\]
In this case the second moment of the estimator is
\begin{align*}
  &\mathcal{Q}^{\e}(0,0,u^\e) = \\
  &\E\left[ \exp\left( -\frac{1}{\epsilon} \int_0^{\tau^\e} |u^\e(s)|_H^2 ds - \frac{2}{\sqrt{\e}} \int_0^{\tau^\e} \left<u^{\e}(s), d\bar{w}(s) \right>_H  \right) \mathbbm{1}_{\{\tau^\e\leq T\}} \right].
\end{align*}

By Girsanov's formula,
\[\frac{d\tilde{\Pro}^\e}{d\bar{\Pro}^\e}=\exp\left( -\frac{2}{\sqrt{\e}} \int_0^{\tau^\e}\left<u^\e(s),d\bar{w}(s)\right>_H -\frac{2}{\e} \int_0^{\tau^\e} |u^\e(s)|_H^2 ds \right)\]
is a change of measure under which
\[\tilde{w}(t) =  \bar{w}(t) + \frac{2}{\sqrt{\e}}\int_0^t u^\e(s)ds = w(t) + \frac{1}{\sqrt{\e}} \int_0^t u^\e(s)ds.\]
is a cylindrical Wiener process. In this way, $\mathcal{Q}^{\e}(0,0,u^\e)$ can be represented as
\[\mathcal{Q}^{\e}(0,0,u^\e) = \tilde{\E}^\e\left[\exp\left(\frac{1}{\e}\int_0^{{\tau}^\e}|u^\e(s)|_H^2 ds \right)\mathbbm{1}_{\{{\tau}^\e\leq T\}}\right]\]
and
\[d{X}^\e(t) = \left[A{X}^\e(t) - Bu^\e(t)  \right]dt + \sqrt{\e}Bd\tilde{w}(t).\]

We will focus in this paper on the case where $u^{\e}(s)=u^\e(s, X^\e(s))$ are feedback controls. By a variational principle (see \cite{bdm-2008}) and by calculations similar  to those in \cite{dsz-2014},
 \begin{equation} \label{eq:var-principle}
  -\e \log \left( \mathcal{Q}^{\e}(0,0,u^\e)\right) = \inf_{v\in \mathcal{A}} \E \int_0^{\hat{\tau}^{v,\e}}\left( \frac{1}{2}|v(s)|_H^2 ds -|u^\e(s,\hat{X}^{v,\e}(s))|_H^2 \right)ds
  \end{equation}
 where $\hat{X}^{v,\e}$ solves
 \begin{equation} \label{Eq:HatXprocess}
   d\hat{X}^{v,\e}(t) = [A\hat{X}^{v,\e}(t) +Bv (t) - Bu^\e(t,\hat{X}^{v,\e}(t))]dt + \sqrt{\e}BdW(t),
 \end{equation}
 and $\mathcal{A}$ is the set of $\mathcal{F}_t$-adapted $H$-valued processes  for which
 \[\hat{\tau}^{v,\e} = \inf\{t>0: |\hat{X}^{v,\e}(t)|_H\geq L\} \leq T \text{ with probability one},\]
 and
 \[\E\int_0^{\hat{\tau}^{v,\e}} |v(s)|_H^2 ds < +\infty.\]

By the related work in the finite dimensional case, \cite[Section 3]{dsz-2014}, we know that asymptotically optimal changes of measure (i.e. with minimum variance as $\e\rightarrow 0$) can perform rather poorly in practice (i.e. for small but fixed $\e>0$). The reason for this behavior is the role of prefactors that can be rather significant if the change of measure does not account for them.
We wish to estimate
\[\theta^\e(x,T) = \Pro \left(  \tau^\e_x \leq T  \right),\]
where
\[\tau^\e_{x} = \inf\{t>0: |X^\e_x(t)|_H \geq L \}.\]
This is the probability that the $L^2(\mathcal{O})$ norm of $X^\e_x$ exceeds $L$ before time $T$.

 What we intend to show below is that the changes of measure that work well in one dimension do not work  well in the infinite dimensional case (or in high finite dimensional cases). Of course, this has to do with the curse of dimensionality. Let us define an object called the quasipotential $V: H \to [0,+\infty]$ by
\begin{equation}
  V(x) = \left|B^{-1}(-A)^{1/2}x\right|_H^2,
\end{equation}
using the convention that if $x \not \in B(-A)^{-1/2}(H)$, then $\left|B^{-1}(-A)^{1/2}x \right|_H^2 = +\infty.$ In \cite{dsz-2014}, importance sampling schemes for the one-dimensional case were designed using the one-dimensional quasipotential. What we shall show below is that a direct analogue of the finite dimensional approach to the infinite dimensional setting is not possible.

Consider the functional $U: H \to \mathbb{R}$ defined by
\begin{equation}
  U(x) = \frac{\alpha_1}{\lambda_1^2} L^2 - |B^{-1}(-A)^{1/2}x|_H^2.
\end{equation}
The Fr\'echet derivative of $U$ is
\[DU(x) = 2B^{-2}A x.\]
Let us define $B^{\star}$ to be the adjoint operator to $B$ and set
\[u(x) = -B^{\star} DU(x) = -2B^{-1}Ax .\]

This $U$ is chosen because it is a solution to the equation
\[
\mathbb{H}(x,DU(x)) = 0\]
where
 \begin{equation}
   \mathbb{H}(x,p) = \left<Ax,p\right>_H - \frac{1}{2}|B^{\star} p|_H^2.
 \end{equation}

By an infinite dimensional analogue to the arguments in \cite[Lemma A.1]{dsz-2014}, for any sufficiently regular function $Z$,
\begin{align} \label{eq:Q-eps-lower-bound}
&-\e \log(\mathcal{Q}^{\e}(0,x,u)) \geq \inf_{\{v: \hat{\tau}^{v,\e} \leq T\}}\Bigg( 2 Z(0,x) -2 \E Z(\hat{\tau}^{v,\e}, \hat{X}^{v,\e}(\hat{\tau}^{v,\e}))\nonumber\\
 &\qquad+ \E \left[\int_0^{\hat{\tau}^{v,\e}} 2 \mathcal{G}^\e[Z](s,\hat{X}^{v,\e}(s)) ds \right.\nonumber\\
 &\qquad\quad\left.-\int_0^{\hat{\tau}^{v,\e}} |B^{\star} DZ(s,\hat{X}^{v,\e}(s)) - B^{\star} DU(\hat{X}^{v,\e}(s))|_H^2 ds \right] \Bigg).
 \end{align}

 In the above expression,
 \begin{equation}\label{eq:Q-eps}
   \mathcal{G}^\e[Z](t,x) = \frac{\partial Z}{\partial t}(t,x) + \mathbb{H}(x,D_xZ(x)) + \frac{\e}{2} \text{Tr}(BB^\star D^2_xZ(t,x)).
 \end{equation}

 Notice that the expression for $\mathcal{G}^\e[Z]$ contains the term $\text{Tr}BB^\star D^2_xZ$. The presence of this term shows that the quasipotential cannot in general be used in an importance sampling scheme. By that we mean that if $Z(x) = U(x) = \frac{\alpha_1}{\lambda_1^2} L^2-|B^{-1}(-A)^{\frac{1}{2}}x|_{H}^{2}$ then $\mathcal{G}^\e[Z] = \e \text{Tr}BB^\star B^{-2}A = \e \text{Tr}A = -\infty$. This is clearly an issue, since then the upper bound for the second moment is $\infty$, which is useless information.

 But the situation is even worse than this. Because $H$ is infinite dimensional, unbounded operators like $A$ are not the only problem. In fact, if $\text{Tr}(BB^\star) =+\infty$, then any $Z(x)$ that is radially symmetric and has a negative second derivative has the property that $\mathcal{G}^\e[Z](x) = -\infty$. By radially symmetric, we mean that there exists $\varphi: \mathcal{R} \to \mathcal{R}$ such that $Z(x) = \varphi(|x|_H)$. In the radially symmetric case,
 \[\frac{\partial Z}{\partial x_k}  = \frac{x_k}{|x|_H} \varphi'(|x|_H), \ \ \frac{\partial^2 Z}{\partial x_k^2} = \frac{x_k^2}{|x|_H^2} \varphi''(|x|_H) + \left(\frac{1}{|x|_H^2} - \frac{x_k^2}{|x|_H^3} \right)\varphi'(|x|_H) .\]
 Therefore,
 \[\text{Tr} D^2_xZ(x) = \sum_{k=1}^\infty \frac{\partial^2 Z}{\partial x_k^2} = \varphi''(|x|_H^2) - \frac{1}{|x|_H}\varphi'(|x|_H) + \infty \frac{\varphi'(|x|_H)}{|x|_H}.\]
 In particular, if $Z(x) = cL^2- c|x|_H^2$ for some constants $L$ and $c$, then
 \[\Tr BB^\star D^2_x Z(x) = -2c\Tr BB^\star= -\infty.\]

 We need to find a control that has a finite trace second derivative. It turns out that the correct control is a projection of the control we described based on the unbounded operator $A$. Instead, we consider the functional
  \[U(x) = \frac{\alpha_1}{\lambda_1^2}L^2 - \frac{\alpha_1}{\lambda_1^2}\left<x,e_1\right>_H^2.\]
  and the feedback control
  \[u(x) = -B^{\star} DU(x) = \frac{2\alpha_1}{\lambda_1}\left<x,e_1\right>_He_1.\]
  This solves our issues with the trace of the second-derivative but introduces other problems into our analysis. Notice that in \eqref{eq:Q-eps-lower-bound}, we wish to maximize $Z(0,y) - Z(\hat{\tau}^{v,\e}, \hat{X}^{v,\e}(\hat{\tau}^{v,\e}))$. We want $Z(t,x)$ to be as large as possible for $|x|_H < L$, but as small as possible on the boundary $|x|_H = L$. If $x=0$, then $Z(0,0) = \frac{\alpha_1}{\lambda_1^2}L^2$. Unfortunately, depending on the location of $\hat{X}^{v,\e}(\hat{\tau}^{v,\e})$,  $Z(\hat{\tau}^{v,\e},\hat{X}^{v,\e}(\hat{\tau}^{v,\e}))$ may be as small as $0$ at $\pm Le_1$ or as big as $\frac{\alpha}{\lambda_1^2}L^2$ if the terminal point is orthogonal to $e_1$. In order for \eqref{eq:Q-eps-lower-bound} to provide us with a useful estimate, we need for $\hat{X}^{v,\e}(\hat{\tau}^{v,\e})$ to be near $\pm L e_1$. The results of the next two sections show that this is indeed the case.

\section{Calculus of variations in the $\e=0$ case}\label{S:CalculusOfVariationsLinearProblem}
\label{sec:calcofvar}
In this section we consider the deterministic controlled process
\[dX^v(t) = [AX^v(t) -Bu(X^v(t)) + Bv (t)]dt.\]
where $u(x) = 2\frac{\alpha_1}{\lambda_1}\left<x,e_1\right>_H e_1$ and $v(t)$ is a deterministic control. We will show in later sections that the controls used in the importance sampling schemes will converge to this $u$ uniformly as $\e \to 0$.

Let $\tau = \inf\{t>0: |X^v(t)|_H>L\}$.
Define the functional $I: C([0,T];H) \to \mathbb{R}$
\[I(X^v) = \int_0^\tau \left(\frac{1}{2}|v(s)|_H^2 -|u(X^v(s))|_H^2 \right)ds.\]
We are interested in the minimizers of
\[\inf\{I(X^v): v \in L^2([0,T];H)\}.\]
Our goal is to show that the minimizers of $I$ over all trajectories 
have the property
that $X^v(\tau)$ is near $\pm L e_1$.

\begin{Lemma}
If $y\in C([0,T];H)$ is absolutely continuous such that $T = \inf\{t>0: |y(t)|_H=L\}$ and $I(y)<+\infty$, then $y$ is weakly differentiable and
\begin{align} \label{eq:I-explicit}
I(y) &= \frac{\alpha_1}{\lambda_1^2} \left<y(T), e_1\right>_H^2 - \frac{\alpha_1}{\lambda_1^2} \left<y(0),e_1\right>_H^2+ \int_0^T \left(\frac{1}{2}|B^{-1}\left(\dot{y}(s) - Ay(s)|_H^2\right)\right) ds.
\end{align}
In the above expression, $\dot{y} = \frac{\partial y}{\partial t}$ in the weak sense.
\end{Lemma}
\begin{proof}
  Let $y(t)=X^v(t)$ and for any $k\geq 1$, let $y_k(t) = \left<y(t),e_k \right>_H$ and $v_k(t) = \left<v(t),e_k\right>_H$. Because $A$ and $Q$ are diagonalized by the $\{e_k\}$ basis, $y_k$ solves
  \[y_k(t) =  \int_0^t \big( -\alpha_k y_k(s)- \left<Bu(y(s)),e_k\right>_H+ \lambda_k  v_k(s)\big)ds. \]
  Each $y_k$ is differentiable. Consequently $y(t)$ is weakly differentiable in the sense that for any test function $h \in H$, $\left<y(t),h\right>_H$ is differentiable and
  \[\frac{d}{dt} \left<y(t),h\right>_H = \sum_{k=1}^\infty \dot{y}_k(t)\left<h,e_k\right>_H.\]
  Furthermore, $v_k(t) = \frac{1}{\lambda_k} \left(\dot{y}_k(t)  + \alpha_k y_k(t) + \left<B u(y(s)),e_k\right>_H\right)$ and
  \[v(t) =B^{-1}\left( \dot{y}(t) - Ay(t) + Bu(y(t))\right).\]
   Thus,
  \[I(y) = \int_0^T \left(\frac{1}{2}\left|B^{-1}\left(\dot{y}(s) - Ay(s) + Bu(y(s))\right)\right|_H^2 - |u(y(s))|_H^2 \right)ds.\]
  We assumed in the statement of the Lemma that $I(y)<+\infty$.
  Therefore, with some simplifications,
  \begin{align*}
    I(y) = \int_0^T \Big( &\frac{1}{2}\left|B^{-1}\left(\dot{y}(s) - Ay(s)\right)\right|_H^2 + \left<u(y(s)),B^{-1}\dot{y}(s)\right>_H \\
    &- \frac{1}{2} |u(y(s))|_H^2 + \left<-B^{-1}A y(s), u(y(s))\right>_H\Big)ds.
  \end{align*}
  The result follows because $\int_0^T \left< B^{-1}\dot{y}(t), u(y(t))\right>_H dt = \frac{\alpha_1}{\lambda_1^2}\left<y(T),e_1\right>_H^2-\frac{\alpha_1}{\lambda_1^2}\left<y(0),e_1\right>_H^2$ and because $\left<-B^{-1}A y(s), u(y(s))\right>_H = \frac{\alpha_1}{\lambda_1} \left< y(s), u(y(s))\right>_H$. Finally, because $u(x) = 2\frac{\alpha_1}{\lambda_1}\left<x,e_1\right>_H e_1$, we obtain
  \[-\frac{1}{2}|u(y(s))|_H^2 + \frac{\alpha_1}{\lambda_1}\left< y(s), u(y(s))\right>_H = 0,\]
  concluding the proof of the lemma.
\end{proof}

Now we find the minimizers of $I$. First, we fix $T >0$ and $z \in H$, and we find the minimizer of $I$ over all $y \in C([0,T];H)$ such that $y(T) =z$. Then we minimize over $z$ and $T$.

\begin{Lemma} \label{lem:I-min-k-geq-2}
  Fix $T>0$ and $z \in H$. If $y^{*} \in C([0,T];H)$ is such that $y^{*}(0)=0$, $y^{*}(T)=z$ and
  \[I(y^{*}) = \inf\{I(y): y(0) = 0, y(T)=z\}, \]
  then
  \[y^{*}_k(t): = \left<y^{*}(t), e_k\right> = \frac{z_k(e^{-\alpha_k(T-t)} - e^{-\alpha_k(t+T)})}{1 - e^{-2\alpha_k T}}\]
  and
  \[\frac{1}{2}\int_0^T \left<v(s),e_k\right>_H^2 ds = \frac{\alpha_kz_k^2}{\lambda_k^2\left(1-e^{-2\alpha_kT}\right)}.\]
\end{Lemma}
\begin{proof}
  By calculus of variations, the Euler-Lagrange equation for $I(y)$ in \eqref{eq:I-explicit} is
  \[\ddot{y}(t) = A^2 y(t)\]
  and the solution is given above. Then
  \[\left<v(t),e_k\right> = \frac{1}{\lambda_k}\left(\dot{y}(t) +\alpha_k y(t)\right) = \frac{2 z_k \alpha_k e^{-\alpha_k(T-t)}}{\lambda_k\left(1 - e^{-2\alpha_kT}\right)}.\]
  The lemma then follows by integrating the square of this.
\end{proof}

Let $y^*$ be as in Lemma \ref{lem:I-min-k-geq-2}. Then by \eqref{eq:I-explicit},
\begin{equation} \label{eq:explicit-I}
    I(y^*) = \frac{\alpha_1}{\lambda_1^2} z_1^2 + \sum_{k=1}^\infty \frac{\alpha_k z_k^2}{\lambda_k^2\left(1 - e^{-2\alpha_k T}\right)}
  \end{equation}
  where $z_k = \left<z, e_k\right>_H$.

In this way, for any fixed $T$ equation \eqref{eq:explicit-I} is a quadratic form in $z$. We can show that the minimizer $y^*$ points only in the $e_1$ direction by studying the eigenvalues.
Let
\begin{align}
  \label{eq:varphi1}&\varphi_1(T) = \frac{\alpha_1 }{\lambda_1^2\left(1- e^{-2\alpha_1T}\right)} +\frac{\alpha_1}{\lambda_1^2},  \\
  \label{eq:varphi2}&\varphi_k(T) = \frac{\alpha_k}{\lambda_k^2\left(1 - e^{-2\alpha_kT}\right)}, \hspace{1cm} k \geq 2
\end{align}

If $\varphi_1(T) < \varphi_k(T)$ for all $k \geq 2$, then the minimum occurs at $z = \pm L e_1$ and by Lemma \ref{lem:I-min-k-geq-2},
\[\left<y^\star(t), e_k \right>_H = 0 \text{ for all } k \geq 2.\]

Now we describe the conditions on the eigenvalues $\lambda_k$ and $\alpha_k$, and the time horizon $T$ that guarantee that the deterministic control problem exits $D = \{x \in H: |x|_H \leq L\}$  at $\pm L e_1$. If $\lambda_1 \geq \lambda_k$ and $3\alpha_1<\alpha_k$ for all $k \in \mathbb{N}$, then we can guarantee that for all time horizons $T>0$, the minimal control problem exits at $\pm L e_1$. If at least one of the $\lambda_k> \lambda_1$, then for small $T$, the minimal control system will not exit near $\pm L e_k$. If however, $2\frac{\alpha_1}{\lambda_1^{2}} < \frac{\alpha_k}{\lambda_k^2}$ we can always find $T_0>0$, such that for all $T>T_0$, the minimal control problem exits at $\pm L e_1$.

\subsection{Minimal control problem if $\lambda_1 \geq \lambda_k$ and $3\alpha_1< \alpha_k$ for all $k \geq 2$}
\begin{assumption} \label{assum:lambda_k-less-than}
  Assume that for all $k \geq 2$, $\lambda_1\geq \lambda_k$ and $3\alpha_1< \alpha_k$.
\end{assumption}
\begin{Theorem} \label{thm:min-traj-e1}
  Assume that Assumption \ref{assum:lambda_k-less-than} holds. Then for any fixed $T>0$, there exists a minimizer $y^{*} \in C([0,T];H)$ satisfying
  \[I(y^{*}) = \inf\{I(y): y \in C([0,T];H), y(0)=0, \ \exists \ 0\leq t\leq T, |y(t)|_H=L\}.\]
  The minimizer $y^*$ has the property that $|y^*(t)|_H< L$ for all $t<T$ and $|y^*(T)|_H = L$. Furthermore, $y^{*}$ only points in the $e_1$ direction. That is, $\left<y^{*}(t),e_k\right>_H = 0$ for $k \geq 2$.
\end{Theorem}
\begin{proof}
  We want to show that $\varphi_1(T)< \varphi_k(T)$ for all $T>0$ and $k\geq2$. Let us define
  \begin{align*}
  \psi_{k}(T;\lambda_{1},\lambda_{k})&=\varphi_{1}(T)-\varphi_{k}(T)=\frac{\alpha_1 }{\lambda_1^2\left(1- e^{-2\alpha_1T}\right)} +\frac{\alpha_1}{\lambda_1^2}-\frac{\alpha_k}{\lambda_k^2\left(1 - e^{-2\alpha_kT}\right)}\nonumber\\
  &=\frac{1}{\lambda_{k}^{2}}\left[\alpha_1\frac{\lambda_k^{2} }{\lambda_1^2}\frac{2- e^{-2\alpha_1 T}}{1- e^{-2\alpha_1 T}}-\alpha_k\frac{1}{1 - e^{-2\alpha_k T}}\right]\nonumber
  \end{align*}

  Since $\lambda_1 \geq \lambda_k$ for all $k \geq 2$, we get that $\frac{\lambda_k^{2} }{\lambda_1^2}\leq 1$ and a sufficient condition for $\psi_{k}(T;\lambda_{1},\lambda_{k})<0$ is
  \[\psi_{k}(T;1,1) = \frac{\alpha_1}{(1-e^{-2\alpha_1 T})} + \alpha_1 - \frac{\alpha_k}{(1-e^{-2\alpha_k T})} <0.\]

  We calculate that the previous display holds for small $T$ because
  \begin{align*}
   &\lim_{T\to 0}\psi_{k}(T;1,1)  = \frac{3\alpha_1}{2} - \frac{\alpha_k}{2}.
  \end{align*}
  This is negative as long as $\alpha_k > 3\alpha_1$.
  We can compare the derivatives in $T$ to show that
   \[\psi^{'}_{k}(T;1,1) = \frac{-2\alpha_1^2 e^{-2\alpha_1 T}}{(1-e^{-2\alpha_1 T})^2} + \frac{2\alpha_k^{2} e^{-2\alpha_k T}}{(1-e^{-2\alpha_k T})^2}\]

   For any fixed $T>0$, the function $x \mapsto \frac{x^2 e^{-2Tx}}{(1-e^{-2Tx})^2}$ is decreasing. Therefore, $\psi^{'}_{k}(T;1,1)<0$ for all $T$ and because $\lim_{T \to 0} \psi_{k}(T;1,1) <0$, this implies that $\psi_{k}(T;1,1)<0$ for all $T>0$. The minimal eigenvalue is $\varphi_1(T)$.

  Then if we take the infimum over all $|z|_H=L$ in \eqref{eq:explicit-I}, the minimum occurs when $z_1 = \pm L e_1$, and $z_k=0$ for all $k\geq 2$. This suggests that the minimizing control only pushes in the $e_1$ direction. Furthermore, $I(y^\star)$ decreases as $T$ increases.
\end{proof}

\begin{Remark}
  In the case where $A=\Delta$ is the Laplace operator endowed with Dirichlet boundary conditions on the spatial domain [0,1] and $B=I$ is the identity, $\alpha_k = \pi^2 k^2$. Consequently, $4 \alpha_1 = \alpha_2 < \alpha_3< ...$. So the Laplace operator satisfies the conditions of the previous theorem.This gap does not exist for heat equations in dimension $d\geq 2$. However, in the next section we describe a relationship between the $\alpha_k$ and $\lambda_k$ that imply minimizers exit in the $e_1$ direction even for the multidimensional heat equation.  
\end{Remark}

\subsection{Minimizing control problem for the general case}
\begin{assumption} \label{assum:lambda-general}
  Assume that for all $k \in \mathbb{N}$, $2\frac{\alpha_1}{\lambda_1^{2}} < \frac{\alpha_k}{\lambda_k^2}$.
\end{assumption}

\begin{Remark}
  Notice that Assumption \ref{assum:lambda_k-less-than} implies Assumption \ref{assum:lambda-general}.
\end{Remark}

In the case where there exists $k \in \mathbb{N}$ such that $\lambda_k > \lambda_1$, it is impossible for $e_1$ to be the minimal exit direction for all $T>0$. Intuitively, this is a consequence of the fact that over short time periods, the noise is much more powerful than the dissipation. Over longer time periods, the dissipation has an important effect. Let $\varphi_1(T)$ and $\varphi_k(T)$ be given as in \eqref{eq:varphi1}-\eqref{eq:varphi2}.

\begin{Lemma}\label{L:Case2_behaviorAtInfty}
  If there exists a $k \in \mathbb{N}$ for which $\lambda_k>\lambda_1$, then for $T$ close to $0$, $\varphi_k(T)< \varphi_1(T)$ and the minimal exit direction is not $e_1$.
\end{Lemma}

\begin{proof}
  Recall the definition $\psi_{k}(T;\lambda_{1},\lambda_{k})=\varphi_{1}(T)-\varphi_{k}(T)$. Straightforward calculations show that in this case
  \[\lim_{T \to 0} \psi_{k}(T;\lambda_{1},\lambda_{k}) = +\infty.\]
  This means that for small $T$, $\varphi_k(T) < \varphi_1(T)$, implying that in this case and for small $T$ the minimal exit direction is not $e_1$.
\end{proof}

If $\lambda_1$ is not the maximum eigenvalue of $Q$, then over short time periods the minimal control problem will not exit in the $e_1$ direction. Over longer time periods, Assumption \ref{assum:lambda-general} is a sufficient condition for a spectral gap to exist.

\begin{Theorem} \label{thm:min-traj-long-time}
  Assume that Assumption \ref{assum:lambda-general} holds. Then there exists a time $T_0>0$
  such that for any $T>T_0$, there exists a minimizer $y^{*} \in C([0,T];H)$ satisfying
  \[I(y^{*}) = \inf\{I(y): y \in C([0,T];H), y(0)=0, \ \exists \ 0\leq t\leq T, |y(t)|_H=L\}.\]
  The minimizer $y^*$ has the property that $|y^*(t)|_H< L$ for all $t<T$ and $|y^*(T)|_H = L$. Furthermore, $y^{*}$ only points in the $e_1$ direction. That is, $\left<y^{*}(t),e_k\right>_H = 0$ for $k \geq 2$.
\end{Theorem}

\begin{proof}
  It is sufficient to prove that there exists $T_0\geq 0$
  such that for all $T>T_0$ and all $k \in \mathbb{N}$, $\varphi_1(T) < \varphi_k(T)$.
  For any fixed $k$,
  \[\lim_{T \to +\infty} \left(\varphi_1(T) - \varphi_k(T)\right) = 2\frac{\alpha_1}{\lambda_1^2}- \frac{\alpha_k}{\lambda_k^2}.\]
  By assumption this is negative. Therefore, there exists a finite $T_k = \sup\{T>0: \varphi_1(T) - \varphi_k(T) \geq 0\}<\infty$. If $\varphi_1(T) - \varphi_k(T) \leq 0$ for all $T>0$, we define $T_k = 0$.

  A couple of useful remarks are in order here. First, notice that if $\lambda_k \leq \lambda_1$ and $3\alpha_1< \alpha_k$, then $T_k=0$ by Theorem \ref{thm:min-traj-e1}. Second, if $k$ is such that $\lambda_k > \lambda_1$ then simple algebra shows that  the function $T\mapsto \varphi_1(T) - \varphi_k(T)$ is decreasing, which combined with Assumption \ref{assum:lambda-general} and Lemma \ref{L:Case2_behaviorAtInfty}, imply that the specific $T_{k}$ is actually the unique solution to $\varphi_1(T) - \varphi_k(T)=0$.

  We define $T_0 = \sup_{k} T_k$ and prove the theorem by showing that $T_0$ is finite. We actually will prove something stronger, namely that $T_k \to 0$.

  By definition, for any $k \in \mathbb{N}$ such that $T_k>0$,
  \[\frac{\alpha_1}{\lambda_1^2( 1- e^{-2\alpha_1 T_{k}})} + \frac{\alpha_1}{\lambda_1^2} = \frac{\alpha_{k}}{\lambda_{k}^2(1-e^{-2\alpha_{k} T_{k}})}\]
  By Assumption \ref{assum:eigens}, $\frac{\alpha_k}{\lambda_k^2} \to +\infty$.
  If there existed a subsequence satisfying $T_{n_k} > \delta>0$, then the left hand side of the above equation would be bounded but the right hand side would converge to infinity. This would be a contradiction and consequently $T_k \to 0$.

  Consequently, there is a $T_0$ such that for all $T>T_0$, the spectral gap exists and the minimal control problem on the time horizon $T$ exits $\{x \in H: |x|_H \leq L\}$ at $\pm L e_1$.
\end{proof}

\begin{Remark}\label{R:MultiDcase}
Numerical evaluations of $\psi_{k}(T;\lambda_{1},\lambda_{k})$ have shown that in the case $\lambda_{k}-\lambda_{1}<-\eta<0$ for an appropriately large $\eta>0$ such that Assumption \ref{assum:lambda-general} holds one actually has that $T_{0}=0$. This implies that if the gap $\lambda_{k}-\lambda_{1}$ is negative and large enough then the condition $3\alpha_{1}<\alpha_{k}$ is not necessary in order to guarantee that $\varphi_1(T) - \varphi_k(T) < 0$ for all $T>0$, implying that  $y^{*}$ only points in the $e_1$ direction. However, we had been unable to find exactly how large the gap $\lambda_{k}-\lambda_{1}$ should be relative to the gap of $\alpha_{1}-\alpha_{k}$ in order for $T_{0}=0$ to be true. Despite that, Theorem \ref{thm:min-traj-long-time} does cover the case of a multidimensional heat equation as long as Assumption \ref{assum:lambda-general} holds.
\end{Remark}

\section{Weak compactness of minimizing trajectories}\label{S:WeakCompactness}
Let $u^\e$ be a sequence of feedback controls that are bounded by and converge uniformly on bounded subsets of $H$ to $u(x) = \frac{2\alpha_1}{\lambda_1}\left<x,e_1\right>_H$ given in the previous section. In practice $u^\e(x) = -B^{\star} DU^{\delta}(x)$ with $\delta=\delta(\e)$ for both of the two exponential convolution schemes outlined in the next section.
\begin{Theorem} \label{thm:conv-of-endpoint}
  Assume that either Assumption \ref{assum:lambda_k-less-than} or Assumption \ref{assum:lambda-general} holds. Let $T_0=0$ if Assumption \ref{assum:lambda_k-less-than} holds or let $T_0$ be as in Theorem \ref{thm:min-traj-long-time} if Assumption \ref{assum:lambda-general} holds. Fix a time horizon $T>T_0$ and let $\mathcal{Q}^{\e}(0,0,u^\e)$ be the variance of the importance sampling estimator \eqref{eq:variance-est}. Let $v^\e \in \mathcal{A}$ be approximate minimizers to the variational problem in the sense that
  \begin{align} \label{eq:ve-assump}
    -\e \log (\mathcal{Q}^{\e}(0,0,u^\e))
    &= \inf_{v \in \mathcal{A}, \hat{\tau}^{v,\e}\leq T} \E \left[\int_0^{\hat{\tau}^{v,\e}}\left( \frac{1}{2}|v(s)|_H^2 -|u^\e(\hat{X}^{v,\e}(s))|_H^2 \right)ds \right]\nonumber\\
    &\geq \E \left[\int_0^{\hat{\tau}^{v^\e,\e}}\left( \frac{1}{2}|v^\e(s)|_H^2 -|u^\e(\hat{X}^{v^\e,\e}(s))|_H^2 \right)ds \right] -\e^{2}
  \end{align}
  where
  \[\hat{\tau}^{v,\e} = \inf \{t>0: |\hat{X}^{v,\e}|_H \leq L\}.\]
   Then as $\e \to 0$,
   \begin{equation} \label{eq:conv-of-endpoint}
     \E\left<\hat{X}^{v^\e,\e}(\hat{\tau}^{v^\e,\e}),e_1\right>_H^2 \to L^2.
   \end{equation}
   In the above equations
   \[d\hat{X}^{v^\e,\e}(t) = \left[A\hat{X}^{v^\e,\e}(t) - Bu(\hat{X}^{v^\e,\e})+Bv ^\e(t)\right]dt + \sqrt{\e}Bdw(t).\]
\end{Theorem}

Theorem \ref{thm:conv-of-endpoint} says that the solution to the minimal variational control problem exits the ball of radius $L$ near the points $\pm L e_1$. The plan for proving Theorem \ref{thm:conv-of-endpoint}  is as follows. First, we prove that the family $\{v^\e\}$ is tight in the sense that a subsequence converges in distribution in the weak topology on $L^2([0,T];H)$. Then this implies the tightness in law of the processes $\{\hat{X}^{v^\e,\e}\}$ in the topology of $C([0,T];H)$ to the minimizing trajectory described in Section \ref{sec:calcofvar}. Then we prove that the stopping times $\hat{\tau}^{v^\e,\e}$ converge to $T$, and finally we can prove the theorem.

The following large deviations principle is an immediate consequence of \cite{Fleming1978} and \cite{bdm-2008}.
\begin{Lemma}[Large Deviations Principle] \label{lem:LDP}
 The limit
 \begin{equation} \label{eq:LDP}
   G_T(u): = \lim_{\e \to 0} -\e \log\left(\mathcal{Q}^{\e}(0,0,u^\e) \right)
 \end{equation}
 exists and is equal to the variational problem with $\e=0$. That is
 \begin{equation} \label{eq:determ-control}
   G_T(u) = \inf_{v \in L^2([0,T];H)}  \int_0^{\hat{\tau}^{v,0}} \left( \frac{1}{2}|v(s)|_H^2 ds - |u^\e(s,\hat{X}^{v,0}(s))|_H^2 \right)ds
 \end{equation}
 which is the problem studied in Section \ref{sec:calcofvar}.
\end{Lemma}

A consequence of the large deviations principle is that $\{v^\e\}$ is bounded in the $L^2(\Omega\times[0,T];H)$ norm.
\begin{Lemma}
  There exists $\e_0>0$ and $C>0$ such that for all $0<\e<\e_0$,
  \[\E\int_0^{\hat{\tau}^{v^\e,\e}}|v^\e(s)|_H^2 ds \leq C\]
  where the $v^\e$ satisfy the assumptions of Theorem \ref{thm:conv-of-endpoint}.
\end{Lemma}
\begin{proof}
  By the large deviations principle (Lemma \ref{lem:LDP}), there exists $\e^{'}_0>0$ such that
  for all $0<\e<\e^{'}_0$,
  \[-\e\log(\mathcal{Q}^{\e}(0,0,u^\e)) \leq C_{0},\]
  for some constant $C_{0}<\infty$.   Then by \eqref{eq:ve-assump},
  \[C_{0} \geq \E\left[\int_0^{\hat{\tau}^{v^\e,\e}}\left(\frac{1}{2}|v^\e(s)|_H^2 - |u(\hat{X}^{v^\e,\e}(t))|_H^2  \right) ds\right]-\e^{2}.\]
  We use the fact that $\sup_{|x|_H\leq L} |u(x)|_H <+\infty$ and potentially choose $\e_{0}<\e^{'}_{0}$ as needed to conclude.
\end{proof}
\begin{Remark}
  Often we will want to study $v^\e$ as $L^2([0,T];H)$ valued random variables. So far the $v^\e$ are only defined on $[0,\hat{\tau}^{v^\e,\e}]$. Without confusion we can always extend $v^\e$ to $[0,T]$ by defining $v^\e(t) = 0$ for all $\hat{\tau}^{v^\e,\e}<t\leq T$.
\end{Remark}

As in \cite{bd-1998,bd-2000,bdm-2008} we define the spaces
\begin{equation}\label{eq:SN-def}
  S_N = \left\{v \in L^2([0,T];H) : \int_0^T |v(s)|_H^2 ds \leq N \right\}
\end{equation}
$S_N$ is a metric space in the topology of weak convergence.

\begin{Lemma} \label{lem:v-tight}
  Without loss of generality we can assume that $v^\e \in S_N$ almost surely and that for any sequence $\e_n \to 0$, there exists a subsequence (also denoted $\e_n$) and a limit $S_N$-valued random variable $v$ for which $v^{\e_n} \to v$ in distribution in $S_N$.
\end{Lemma}
For details about the proof of Lemma \ref{lem:v-tight} see the proof of \cite[Theorem 4.4]{bd-2000}. Next we show that the tightness of $\{v^\e\}$ implies the tightness of $\hat{X}^{v^\e,\e}$ in $C([0,T];H)$.

\begin{Lemma} \label{lem:stoch-conv-bound}
 Define the stochastic convolution
 $\Phi(t) = \int_0^t e^{A(t-s)}BdW(s)$. There exists $C>0$ and $p>1$ such that
 \[\E\sup_{0\leq t\leq T} |\Phi(t)|_H^p \leq CT.\]
\end{Lemma}
The above lemma can be proven using the stochastic factorization method (see \cite{DaP-Z}).

\begin{Lemma} \label{lem:Psi-vn-to-Psi-v}
  If $v_n \to v$ weakly in $L^2([0,T];H)$, and
  \[\Psi_n(t) = \int_{0}^t e^{A(t-s)}Bv _n(s)ds,\]
  then $\Psi_n$ converges in $C([0,T];H)$ to
  \[\Psi(t) = \int_{0}^t e^{A(t-s)}Bv (s) ds.\]
\end{Lemma}

\begin{proof}
  This is a consequence of the Arzela-Ascoli theorem in infinite dimensional spaces. First we prove that for any fixed $t$, $\Psi_n(t) \to \Psi(t)$. For any fixed $t\geq0$, and for $M\geq 0$, by the H\"older inequality,
  \begin{align*}
    &|\Psi_n(t) - \Psi(t)|_H^2 = \sum_{k=1}^\infty \left(\int_{0}^t e^{-2\alpha_k(t-s)}\lambda_k\left<v_n(s)-v(s),e_k\right>_Hds \right)^2\\
    &\leq\quad \sum_{k=1}^M \left(\int_{0}^t e^{-2\alpha_k(t-s)}\lambda_k\left<v_n(s)-v(s),e_k\right>_Hds \right)^2 \nonumber\\
    &\quad+ \frac{\lambda_{M+1}^2}{2 \alpha_{M+1}}\int_{0}^t \left(|v(s)|_H+|v_n(s)|_H \right)^2ds.
  \end{align*}
  The second term can be made arbitrarily small for large $M$, and the finite sum converges to $0$ because of the weak convergence of the $v_n$.
  Furthermore, the $\{\Psi_n\}$ family is uniformly continuous. For any $r<t\leq0$,
  \begin{align}
  |\Psi_n(t) - \Psi_n(r)|_H &\leq \left|\int_{0}^r \left(e^{A(t-s)} - e^{A(r-s)} \right)Bv _n(s)ds \right|_H + \left|\int_r^t e^{A(t-s)}Bv _n(s)ds \right|_H \nonumber\\
  &=:J_1 + J_2. \nonumber
  \end{align}
  We estimate that for $0<\gamma< \frac{1}{2}$,
  \begin{align*}
    &J_1 = \left| (e^{A(t-r)} - I)\int_{0}^r e^{A(r-s)}Bv _n(s)ds \right|_H\\
    &\leq \left\|(e^{A(t-r)} - I) (-A)^{-\gamma}\right\|_{\mathscr{L}(H)} \left|(-A)^\gamma \int_0^r e^{A(r-s)}Bv _n(s)ds \right|_H\\
   &\leq C\left\|(e^{A(t-r)} - I) (-A)^{-\gamma}\right\|_{\mathscr{L}(H)} \int_{0}^r (t-s)^{-\gamma}|Bv _n(s)|_Hds\\
   &\leq  C(t-r)^\gamma \sqrt{\int_{0}^r|v_n(s)|_H^2 ds}.
  \end{align*}
  This converges to 0 as $t-r \to 0$. A simple Cauchy-Schwarz inequality shows
  \[J_2 \leq \sqrt{t-s}\sqrt{\int_{r}^t|v_n(s)|_H^2ds}.\]
  Therefore
  \[|\Psi_n(t) - \Psi_n(r)|\leq C|t-s|^\gamma\]
  and the family is equicontinuous. By the Arzela-Ascoli theorem, a subsequence converges in $C([0,T];H)$.
\end{proof}

\begin{Lemma} \label{lem:X-conv}
  Suppose that $v_n \to v$ in $S_N$ and $\e_n \to 0$. Then
  $\hat{X}^{v^{\e_n},\e_n}$ converges in distribution in $C([0,T];H)$ to $\hat{X}^{v,0}$.
\end{Lemma}
\begin{proof}

  Let $\mathcal{F}:[0,1]\times C([0,T];H) \to C([0,T];H)$ be the mapping
  \[\mathcal{F}(\e,\Psi)(t) = \int_0^t e^{A(t-s)}u^\e(s,\mathcal{F}(\e,\Psi)(s))ds + \Psi(t).\]
  We use the notation that $u^0=u$.
  A standard Gr\"onwall argument shows that $\mathcal{F}$ is well-defined and it is continuous in $\Psi$. $\mathcal{F}$ is also continuous as $\e \to 0$ because $u^\e \to u$ uniformly by assumption.
  Observe that
  \begin{align*}
    &\mathcal{F}(\e,\Psi)(t) - \mathcal{F}(0,\Psi)(t)
    \leq C \left(\int_0^t  |u^\e(s,\mathcal{F}(\e,\Psi)(s))-u(\mathcal{F}(\e,\Psi)(s))|_H ds \right.\\
    &\qquad \left.+ \int_0^t |u(\mathcal{F}(\e,\Psi)(s)) - u(\mathcal{F}(0,\Psi)(s))|_Hds \right) \\
    &\leq Ct |u^\e-u|_{L^\infty} + \int_0^t \|u\|_{\text{Lip}} |\mathcal{F}(\e,\Psi)(s)-\mathcal{F}(0,\Psi)(s)|_H ds.
  \end{align*}
  In the above expression $|u^\e-u|_{L^\infty} = \sup_{|h|_H \leq L} |u^\e(h)-u(h)|_H$ and $\|u\|_{\text{Lip}} = \sup_{|h_1|\leq L, |h_2|\leq L} \frac{u(h_1)-u(h_2)}{h_1-h_2}$.
  By a Gr\"onwall argument, $\mathcal{F}$ is continuous as $\e \to 0$ uniformly in $\Psi$ because $u^\e \to u$ uniformly.

  Using this notation, we get
  \[\hat{X}^{v^{\e_n},\e_n} = \mathcal{F}\left(\e_n, \int_0^\cdot e^{A(\cdot-s)} v^{\e_n}(s) ds + \sqrt{\e_n}\int_0^\cdot e^{A(\cdot-s)}Bdw(s)\right)\]

  We showed in Lemma \ref{lem:Psi-vn-to-Psi-v} that $\int_0^t e^{A(t-s)}v^{\e_n}(s) \to \int_0^t e^{A(t-s)}v(s)$ in distribution in $C([0,T];H)$. It is a consequence of Lemma \ref{lem:stoch-conv-bound} that $\sqrt{\e}\int_0^t e^{A(t-s)}dw(s) \to 0$ in distribution. Therefore, by the continuity of $\mathcal{F}$, $\hat{X}^{v_n,\e_n} \to \hat{X}^{0,v}$ in distribution.
\end{proof}

Now we prove that any limit of $\hat{X}^{v^\e,\e}$ is a distribution that is concentrated on the minimizing trajectories of the deterministic control problem characterized in section \ref{sec:calcofvar}.

\begin{Lemma} \label{lem:conv-to-determ-min}
  Let $v^\e$ be approximate minimizers as in \eqref{eq:ve-assump}. Let $\e_n \to 0$ be a subsequence such that $v^{\e_n}$ converges in distribution in $S_N$ to a limit $v^0$. Such a subsequence exists by Lemma \ref{lem:v-tight}. Then $v^0$ is a distribution that is concentrated on the minimizing controls of the deterministic system. That is
  \begin{align*}
    &\int_0^{\hat{\tau}^{v^0,0}} \left(\frac{1}{2}|v^0(s)|_H^2 - |u(\hat{X}^{v^0,0}(s))|_H^2 \right)ds\\
    &=\inf_{v \in L^2([0,T];H)}\int_0^{\hat{\tau}^{v,0}} \left(\frac{1}{2}|v(s)|_H^2 - |u(\hat{X}^{v,0}(s))|_H^2 \right)ds
  \end{align*}
  with probability one.
\end{Lemma}

\begin{proof}
  By \eqref{eq:ve-assump} and Lemma \ref{lem:LDP},
  \[\limsup_{n \to +\infty} \E\left[ \int_0^{\hat{\tau}^{v^{\e_n},\e_n}} \left(\frac{1}{2}|v^{\e_n}(s)|_H^2 - |u(\hat{X}^{v^{\e_n},\e_n}(s))|_H^2 \right)ds\right] \leq G_T(u). \]
  On the other hand, by the fact that $v^{\e_n} \to v$ in distribution in $S_N$ and by Lemma \ref{lem:X-conv},
  \begin{align*}
    &\liminf_{n \to +\infty} \E\left[ \int_0^{\hat{\tau}^{v^{\e_n},\e_n}} \left(\frac{1}{2}|v^{\e_n}(s)|_H^2 - |u(\hat{X}^{v^{\e_n},\e_n}(s))|_H^2 \right)ds\right]\\
    &\geq \E\left[\int_0^{\hat{\tau}^{v^0,0}} \left(\frac{1}{2}|v^0(s)|_H^2 - |u(\hat{X}^{v^0,0}(s))|_H^2\right)ds\right]
  \end{align*}
  Consequently,
  \begin{align*}
    &\E\left[\int_0^{\hat{\tau}^{v^0,0}} \left(\frac{1}{2}|v^0(s)|_H^2 - |u(\hat{X}^{v^0,0}(s))|_H^2\right)ds\right]\\
    &\leq \inf_{v \in L^2([0,T];H)}\E\int_0^{\hat{\tau}^{v,0}} \left(\frac{1}{2}|v(s)|_H^2 - |u(\hat{X}^{v,0}(s))|_H^2 \right)ds.
  \end{align*}
  But since the right-hand side is the infimum, the limit $v^0$ must attain the infimum with probability 1.
\end{proof}

\begin{Corollary}
  Assume $T>T_0$. Let $v^0$ be as in Lemma \ref{lem:conv-to-determ-min}, then $\hat{\tau}^{v^0,0} = T$ and $\left<\hat{X}^{v,0}(\hat{\tau}^{v,0}), e_1\right>_H^2 = L^2$
\end{Corollary}
This is an immediate consequence of Lemma \ref{lem:conv-to-determ-min} and Theorem \ref{thm:min-traj-e1} or Theorem \ref{thm:min-traj-long-time} which say that the minimizing trajectory only points in the $e_1$ direction and that it exits at time $T$.

\begin{Lemma} \label{lem:endpoint-conv}
  Let $v^{\e_n}$ and $v^0$ be as in Lemma \ref{lem:conv-to-determ-min}. Then
  \[\lim_{\e\rightarrow 0}\E \left<\hat{X}^{v^{\e_n},\e_n}(\hat{\tau}^{v^{\e_n},\e_n}),e_1 \right>^2 = L^2.\]
\end{Lemma}

\begin{proof}
  First we show that $\hat{\tau}^{v^{\e_n},\e_n}$ converges to $T$ in probability. Let $\eta>0$ and notice that
  \[\Pro(\hat{\tau}^{v^{\e_n},\e_n}>T-\eta) = \Pro\left(\sup_{0\leq t \leq T-\eta}|\hat{X}^{v^{\e_n},\e}(t)|_H<L \right).\]
  Then because $\hat{X}^{v^{\e_n},\e_n} \to \hat{X}^{v,0}$ in distribution and the $\sup$ is continuous in that metric,
  \[\lim_{n \to +\infty} \Pro(\hat{\tau}^{v^{\e_n},\e_n}>T-\eta) = \Pro\left(\sup_{0\leq t \leq T-\eta}|\hat{X}^{v^0,0}(t)|_H<L \right) = 1.  \]
  The above formula is a consequence of Theorem \ref{thm:min-traj-e1} or Theorem \ref{thm:min-traj-long-time}, which say that the minimum trajectory satisfies $|\hat{X}^{v^0,0}(t)|_H < L$ for all $t< T$.
  Now because $\hat{\tau}^{v^{\e_n},\e_n}$ converges to $T$  in probability and $\hat{X}^{v^{\e_n},\e} \to \hat{X}^{v^0,0}$ in distribution, it follows that $\hat{X}^{v^{\e_n},\e_n}(\hat{\tau}^{v^{\e_n},\e_n}) \to \hat{X}^{v^0,0}(T) = \pm L e_1$ in distribution.
  The result follows.
\end{proof}

Now we can prove the main theorem of this section.
\begin{proof}[Proof of Theorem \ref{thm:conv-of-endpoint}]
 Let $v^\e$ satisfy \eqref{eq:ve-assump}. Let $\e_n\to0$ be any subsequence. Then by Lemma \ref{lem:v-tight}, there is a further subsequence (relabeled as $\e_n$) for which $v^{\e_n} \to v^0$ in distribution. By Lemma \ref{lem:conv-to-determ-min} $v^0$ is concentrated on the minimizing controls of the deterministic system. By Lemma \ref{lem:endpoint-conv},
 \[\lim_{n \to \infty} \E \left<\hat{X}^{v^{\e_n},\e_n}(\hat{\tau}^{v^{\e_n},\e_n}),e_1 \right>^2 = L^2.\]
\end{proof}

\section{Analysis of importance sampling}\label{S:IS_schemes}

The goal of this section is to discuss construction and theoretical performance of concrete importance sampling schemes. As in the previous section, we assume that either Assumption \ref{assum:lambda_k-less-than} or Assumption \ref{assum:lambda-general} holds. If Assumption \ref{assum:lambda_k-less-than} holds we fix any time horizon $T>0$. If Assumption \ref{assum:lambda-general} holds, then we fix $T>T_0$ where $T_0$ is as in Theorem \ref{thm:min-traj-long-time} and thus $T$ needs to be large enough. This is not a problem for us as we are indeed interested in developing schemes that are stable for large $T$.

  For functions $U(t,x)$ and $Z(t,x)$ let us define the operator
\begin{align*}
\mathcal{G}^{\e}[Z,U](t,x)&=\mathcal{G}^{\e}[Z](t,x)-\frac{1}{2}\left|B^{\star}\left(D_{x}Z(t,x)-D_{x}U(t,x)\right)\right|^{2}_{H}
\end{align*}
where $\mathcal{G}^{\e}[Z](t,x)$ is defined in (\ref{eq:Q-eps}). Then, by \cite[Lemma A.1]{dsz-2014}, we get for $ u^\e(t,x)=-B^{\star} D_{x}U^\e(t,x)$ the non-asymptotic bound
\begin{align} \label{eq:Q-eps-lower-bound2}
-\e \log(\mathcal{Q}^{\e}(0,x,u^\e)) &\geq  \inf_{v \in \mathcal{A}}\Bigg(2 Z(0,x) -2 \E Z(\hat{\tau}^{v,\e}, \hat{X}^{v,\e}(\hat{\tau}^{v,\e})) \nonumber\\
&\qquad + 2 \E \int_0^{\hat{\tau}^{v,\e}}  \mathcal{G}^\e[Z,U^\e](s,\hat{X}^{v,\e}(s)) ds\Bigg).
 \end{align}
where $\hat{X}^{v,\e}$ satisfies (\ref{Eq:HatXprocess}) and $\mathcal{A}$ is the set of adapted $L^2([0,T];H)$ controls for which $\hat{\tau}^{v,\e}\leq T$. The function $u^\e(t,x)=-B^{\star} D_{x}U^\e(t,x)$ is used for the implementation of the scheme, whereas the function $Z(t,x)$ is used for the analysis of the scheme. The bound (\ref{eq:Q-eps-lower-bound2}) holds for any $U^\e$ and $Z$, but in the analysis of the specific schemes considered below we will make specific choices for $U$ and $Z$, also linking them together.

Our goal is to provide implementable importance sampling schemes for which $\mathcal{G}^{\e}[Z,U^\e](t,x)\geq 0$ for all $(t,x)\times[0,T]\times H$ or at least
\[
 2 \E \int_0^{\hat{\tau}^\e}  \mathcal{G}^\e[Z,U^\e](s,\hat{X}^{v,\e}(s)) ds>-Cf(\e).
 \]
  for some function $f(\e)$ such that $\lim_{\e\rightarrow 0}f(\e)=0$ uniformly with respect to $T<\infty$. As it is also discussed in the finite dimensional case of \cite{dsz-2014}, controlling the term $\mathcal{G}^{\e}[Z,U^\e](t,x)$ is vital when it comes to assessing the performance of a given importance sampling scheme. This is no different in the infinite dimensional case and as we also mentioned in  Section \ref{S:IS} if for example we choose $Z(t,x)=U(t,x)= \frac{\alpha_1}{\lambda_k^2}L^{2}-|B^{-1}(-A)^{\frac{1}{2}}x|_{H}^{2}$, then  $\mathcal{G}^\e[Z,U] = \e \text{Tr}A = -\infty$, which implies that in this case we have no control on the performance of the corresponding importance sampling scheme.

Let us set $Z(t,x)=(1-\eta)U(t,x)$ with $\eta\in(0,1)$. Then, straightforward algebra gives
\begin{align}
\mathcal{G}^{\e}[Z,U](t,x)&=\mathcal{G}^{\e}[Z](t,x)-\frac{\eta^{2}}{2}\left|B^{\star} D_{x}U(t,x)\right|^{2}_{H}\nonumber\\
&\geq (1-\eta)\mathcal{G}^{\e}[U](t,x)-\frac{\eta-2\eta^{2}}{2}\left|B^{\star}D_{x}U(t,x)\right|^{2}_{H}\label{Eq:Goperator}
\end{align}

We construct importance sampling schemes based on constructions that exploit different properties of the dynamical system near the attractor and away from it and then combine them in an appropriate smooth way. In particular, for $k_{1},k_{2}\in\mathbb{N}$, if $F_{i}(t,x), i=1,\cdots, k_{1}$ are good change of measure in parts of the phase space away form the rest point while $F_{j}(t,x), j=k_{1}+1,\cdots,k_{2}$ are good changes of measure in parts of the phase space within the neighborhood of the rest point, then we consider $0<\delta\ll 1$ (which is to be chosen) and we define the exponential mollification of $F_{i}(t,x), i=1,\cdots,k_{2}$ (similarly to \cite{dsz-2014})
\begin{align*}
\bar{U}^{\delta}(t,x)&=-\delta\log\left(\sum_{i=1}^{k_{1}}e^{-\frac{F_{i}(t,x)}{\delta}}+\sum_{j=k_{1}+1}^{k_{2}} e^{-\frac{F_{j}(t,x)}{\delta}}\right)
\end{align*}
We notice that  $\lim_{\delta\downarrow 0}\bar{U}^{\delta}(t,x)=F_{1}(t,x)\wedge F_{2}(t,x)\wedge\cdots\wedge F_{k_{2}}(t,x)$. We also notice that the Fr\'{e}chet derivative of the exponential mollification $\bar{U}^{\delta}(t,x)$ is
\begin{align*}
D_{x}\bar{U}^{\delta}(t,x)&=\sum_{i=1}^{k_{2}}\rho_{i}(t,x)D_{x}F_{i}(t,x), \text{ where } \rho_{i}(t,x)=\frac{e^{-\frac{F_{i}(t,x)}{\delta}}}{\sum_{i=1}^{k_{2}}e^{-\frac{F_{i}(t,x)}{\delta}}}
\end{align*}

As it will be discussed in the sequel we choose the functions $F_{i}$ such that the corresponding weight function $\rho_{i}(t,x)\approx 0$ away from the part of the phase space where $F_{i}$ is intended to dominate, whereas $\rho_{i}(t,x)\approx 1$ within the area of the phase space where $F_{i}$ is intended to dominate. In particular the exponential mollification allows for a smooth transition between the regions where $F_{i}$ for $i=1,\cdots,k_{2}$ are supposed to be inducing the desirable change of measure.

Now that we have described the general construction, let us go into specifics for the problem at hand. The results of Sections \ref{S:CalculusOfVariationsLinearProblem} and \ref{S:WeakCompactness} motivate considering importance sampling schemes that only act in the $e_{1}$ direction. By Theorem \ref{thm:min-traj-e1} or Theorem \ref{thm:min-traj-long-time} one expects that in the linear case such schemes work well at least when Assumption \ref{assum:lambda_k-less-than} or Assumption \ref{assum:lambda-general} holds. So, let us consider a change of measure induced by a function $U(t,x)$ such that the control  $u(t,x)=-B^{\star} D_{x}U^\e(t,x)$ only acts in the $e_1$ direction. That is
\[u^\e(t,x) = u^\e(t,\left<x,e_1\right>_H e_1) \text{ and } \left<u^\e(t,x), e_k\right>_H = 0 \text{ for } k\geq 2.\]

We will discuss two different ways to choose the functions $F_{i},i=1,\cdots,k_{2}$ which then form $\bar{U}^{\delta}(t,x)$ that is the basis for defining $U^\e(t,x)$ . The first way is motivated by the one-dimensional construction of \cite{dsz-2014}. The second way is similar to the first one in spirt, but simpler to apply and with comparable performance. Motivated by the results of Sections \ref{S:CalculusOfVariationsLinearProblem} and \ref{S:WeakCompactness}, we choose for both constructions $k_{1}=1$ and
\begin{align*}
F_{1}(t,x)=\frac{\alpha_{1}}{\lambda_1^2}\left(L^{2}-\left<x,e_{1}\right>^{2}_{H}\right)
\end{align*}
which turns out to induce a simple but provably good change of measure away from the rest point. The two different ways that we present differ on what one does in the neighborhood of the rest point, i.e. in the neighborhood of the attractor. Also without loss of generality we set the initial point to be $x=0$.

Before proceeding with the analysis for each of the
schemes, we give the definition of exponential negligibility
\begin{Definition}
\label{D:ExponentialNegligibility} A term is called exponentially negligible if
it is bounded above in absolute value by a quantity of the form $\e
c_{1}e^{-\frac{c_{2}}{\e}}$, where $c_{1}<\infty$, $c_{2}>0$.
\end{Definition}

\textbf{Scheme 1:} Motivated by \cite{dsz-2014}, let us consider the minimization problem
\begin{align*}
V(t,x)&=\inf_{u\in L^{2}([0,T];H): dX(t)=[AX(t)+Bu(t)]dt, X(t)=x\in H, X(T)=z\in H}\left\{\frac{1}{2}\int_{t}^{T}|u(s)|^{2}_{H}ds \right\}
\end{align*}
where $|z|^{2}_{H}\leq L^{2}$. One can solve this variational problem in closed form and get
\begin{align}
V(t,x)&=\frac{1}{2}\left| B^{-1}(-2A)^{1/2}(I-e^{2A(t-T)})^{-1/2}(z-x e^{A(t-T)})\right|^{2}_{H}
\end{align}

Due to the singularities appearing at $t=T$ we next introduce a regularization parameter $M\gg 1$ and consider
\begin{align}
V_{M}(t,x)&=\frac{1}{2}\left| B^{-1}(-2A)^{1/2}((1+M^{-1})I-e^{2A(t-T)})^{-1/2}(z-x e^{A(t-T)})\right|^{2}_{H}
\end{align}

We choose $k_{2}=2$. Projecting  $V_{M}$  down to the $e_{1}$ direction and controlling for the possibility $|\left<z,e_{1}\right>_H|^{2}<L^{2}$, this leads to the following definition for $F_{2}, F_{3}$, which is analogous to the corresponding definitions of \cite{dsz-2014} for the one-dimensional case,
\begin{align}
F_{2}(t,x)&=\frac{\alpha_{1}}{\lambda_1^2\left(\frac{1}{M}+1-e^{2\alpha_{1}(t-T)}\right)}\left(\left<z,e_{1}\right>^{2}_{H}+e^{2\alpha_{1}(t-T)}\left<x,e_{1}\right>^{2}_{H}\right.\nonumber\\
&\qquad\left.-2e^{\alpha_{1}(t-T)}\left<z,e_{1}\right>_{H}\left<x,e_{1}\right>_{H}\right)+\frac{\alpha_{1}}{\lambda_1^2}(L^{2}-\left<z,e_{1}\right>^{2}_{H})\label{Eq:F2_scheme1}
\end{align}
and
\begin{align}
F_{3}(t,x)&=\frac{\alpha_{1}}{\lambda_1^2\left(\frac{1}{M}+1-e^{2\alpha_{1}(t-T)}\right)}\left(\left<z,e_{1}\right>^{2}_{H}+e^{2\alpha_{1}(t-T)}\left<x,e_{1}\right>^{2}_{H}\right.\nonumber\\
&\qquad\left.+2e^{\alpha_{1}(t-T)}\left<z,e_{1}\right>_{H}\left<x,e_{1}\right>_{H}\right)+\frac{\alpha_{1}}{\lambda_1^2}(L^{2}-\left<z,e_{1}\right>^{2}_{H})\label{Eq:F3_scheme1}
\end{align}

As in \cite{dsz-2014} due to the singularities at $t=T$ this scheme needs one more mollification parameter denoted by $t^{*}$ and we finally set $u^\e(t,x)=-D_{x}U^{\delta}(t,x)$ where
\begin{align}
U^{\delta}(t,x)&=\left\{
\begin{array}
[c]{cc}%
F_{1}(x), & t>T-t^{\ast}\\
\bar{U}^{\delta}(t,x), & t\leq T-t^{\ast}%
\end{array}
\right.  , \label{Eq:ControlDesignLinear}%
\end{align}

Then we can establish the following theorem,  whose proof is omitted as it is exactly analogous to that of Theorem 4.7 in \cite{dsz-2014}. We only remark that the statement $\lim_{\e\downarrow 0}\E Z(\hat{\tau}^{v,\e}, \hat{X}^{v,\e}(\hat{\tau}^{v,\e}))=0$ that is needed in the infinite dimensional case that we consider in this paper is a direct consequence of Theorem \ref{thm:conv-of-endpoint} and the definition of the $Z$ function.
\begin{Theorem}
\label{T:Scheme1} Assume that $\delta=2\e,\eta\in(\e
/(\e+\alpha_{1}L^{2}),1/4)$. Let $u^\e(t,x)=-B^{\star} D_{x}U^{\delta}(t,x)$ where $U^{\delta}$ is defined in (\ref{Eq:ControlDesignLinear}). Then up to an
exponentially negligible term in $\e$, we have for $\e$ sufficiently small
\begin{align}
-\e\log \mathcal{Q}^{\e}(0,0;{u}^\e)&\geq 2I_{1}(\e
,\eta,T,|\left<z,e_1\right>_H|^{2}, M)1_{\left\{  T\geq t^{\ast}\right\}  }+2I_{2}(\e
,T)1_{\left\{  T<t^{\ast}\right\} }\nonumber\\
&\quad-\E Z(\hat{\tau}^\e, \hat{X}^\e(\hat{\tau}^\e)),\nonumber
\end{align}
where
\[
I_{1}(\e,\eta,T,|\left<z,e_1\right>_H|^{2},M)=(1-\eta)U^{\delta}(0,0)+\e R(\eta,T,|\left<z,e_1\right>_H|^{2},M).
\]
\[
I_{2}(\varepsilon,T)=\frac{\alpha_{1}}{\lambda_1^2}(L^{2}-T\e)
\]
Here $R(\eta,T,|\left<z,e_1\right>_H|^{2},M )$\footnote{For the exact form of  $R(\eta,T,|\left<z,e_1\right>_H|^{2},M )$ we refer the interested reader to Theorem 4.7 in \cite{dsz-2014}. We do not report it here as the formula is long and not useful for our purposes.} is a negative function that is uniformly bounded in all of its arguments,
\[
U^{\delta}(0,0)\geq\frac{\alpha_{1}}{\lambda_1^2\left(\frac{1}{M}+1-e^{-2\alpha_{1}T}\right)}\left<z,e_{1}\right>^{2}+\frac{\alpha_{1}}{\lambda_1^2}(L^{2}-\left<z,e_{1}\right>^{2})  -\delta\log3
\]
and $\lim_{\e\downarrow 0}\E Z(\hat{\tau}^\e, \hat{X}^\e(\hat{\tau}^\e))=0$.
\end{Theorem}

\begin{Remark}
As mentioned in \cite{dsz-2014} there are natural scalings under which  $\eta\rightarrow0$
and $M\rightarrow\infty$, $\delta\downarrow 0$ as $\e\rightarrow0$. We can set $\delta=2\e$, $M=\e^{-\kappa}$ with $\kappa\in(0,1)$, $t^{*}=-\frac{2\lambda_1^2}{\alpha_{1}}\log\frac{1}{M}$. Also, the value of $|\left<z,e_1\right>_H|^{2}$ is not that important as long as it is of order one and less than $L^{2}$.
If the natural scalings are used then various terms vanish as $\e \rightarrow 0$, and we obtain that
\[
\lim_{\e\downarrow 0}U^{\delta}(0,0)=\alpha_{1}L^{2}+\alpha_{1}|\left<z,e_1\right>_H|^{2}  \frac{e^{-2\alpha_{1}T}}{1-e^{-2\alpha_{1}T}}
\]
uniformly in $T$ as $\e\rightarrow0$.
\end{Remark}

\begin{Theorem} \label{thm:u_eps-to-u-scheme-1}
  Let $u^\e(t,x) = B^{\star} D_xU^{\delta}(t,x)$ where $U^{\delta}$ with $\delta=2\e$ is defined in \eqref{Eq:ControlDesignLinear} with $t^* = -\frac{2\lambda_1^2}{\alpha_1}\log\left(\frac{1}{M}\right)$, then $u^\e$ converges uniformly to $u(x) = \frac{\alpha_1}{\lambda_1} \left<x,e_1\right>_H$.
\end{Theorem}

\begin{proof}
  This is an immediate consequence of the fact that $t^*$ converges to $\infty$. Therefore, $U^{\delta} = F_1$ and $u^\e=u$ for small $\e$.
\end{proof}

\textbf{Scheme 2:} A further analysis of Scheme 1 leads to the conclusion that the role of $F_{2}$ and $F_{3}$ as defined by (\ref{Eq:F2_scheme1})-(\ref{Eq:F3_scheme1}) is to push trajectories very gently outside the area of attraction in a time dependent manner. As we will see in the analysis below and in the numerical simulation results of Section \ref{S:Numerics}, actually doing no change of measure in the neighborhood of zero leads to schemes with comparable performance, but simpler in terms of implementation. In particular, we now set $k_{2}=1$ (i.e. we now need only one function to control the behavior in the neighborhood of the rest point) and we define
\begin{align}
F^{\e}_{2}=\frac{\alpha_{1}}{\lambda_1^2}(L^{2}-\e^{\kappa}), \text{ where }\kappa\in(0,1)
\end{align}

As the analysis below will demonstrate the  term $\e^{\kappa}$ defines the size of the neighborhood  of the attractor outside of which $F_{1}$ takes over. The simulation results indicate that $\e^{\kappa}$ should be neither too large, nor too small. We set
$u^\e(x)=-D_{x}U^{\delta}(x)$ where
\begin{align}
U^{\delta}(x)=-\delta\log\left(e^{-\frac{F_{1}(x)}{\delta}}+e^{-\frac{F^{\e}_{2}}{\delta}}\right)  , \label{Eq:ControlDesignLinear2}%
\end{align}

Lemma \ref{L:BoundGScheme2} takes care of the integral term in the upper bound of the second moment of the estimator and its proof is given in Appendix \ref{A:ProofLemmaBoundGScheme2}. For $0<\eta<1$ to be chosen later, we set $Z =U^{\delta,\eta} = (1-\eta)U^{\delta}$.
\begin{Lemma}\label{L:BoundGScheme2}
Consider the function $U^{\delta}(x)$ of Scheme 2 as defined by (\ref{Eq:ControlDesignLinear2}) and for $\kappa\in(0,1)$ let us consider $\e$ sufficiently small such that $\e^{1-\kappa}\leq \frac{\alpha_{1}}{2\lambda_1^2}$. Then for all $(t,x)\in[0,T]\times H$ we have $\mathcal{G}^{\e}[U^{\delta,\eta},U^{\delta}](x)\geq 0$.
\end{Lemma}

\begin{Theorem} \label{thm:u_eps-to-u-scheme-2}
  Let $u^\e(x) = B^{\star} D_xU^{\delta}(x)$ where $U^{\delta}$ with $\delta=2\e$ is defined in \eqref{Eq:ControlDesignLinear2} then $u^\e$ converges uniformly  to $u(x) = \frac{\alpha_1}{\lambda_1^2} \left<x,e_1\right>_H$ on bounded subsets of $H$.
\end{Theorem}

\begin{proof}
  Notice that
  \[B^{\star} D_xU^\delta(x) = \rho_1^\delta(x) B^{\star} D_x F_1(x) = \rho_1^\delta(x) \frac{\alpha_1}{\lambda_1} \left<x,e_1\right>_H \]
  where
  \[\rho_1^\delta(x) = \frac{e^{-\frac{F_1(x)}{\delta}}}{e^{-\frac{F_1(x)}{\delta}} + e^{-\frac{F^\e_2}{\delta}}}.\]

  Then setting $\delta=2\e$
  \[|u^\e(x) - u(x)| = \left|1-\rho_1^{2\e}(x)\right| \frac{\alpha_1}{\lambda_1}\left|\left<x,e_1\right>_H\right|.\]

  Notice that
  \[\left|1 - \rho_1^{2\e}(x)\right| = \frac{e^{-\frac{F^\e_2}{2\e}}}{e^{-\frac{F_1(x)}{2\e}}+ e^{-\frac{F^\e_2}{2\e}}} = \frac{1}{e^{\frac{F^\e_2 - F_1(x)}{2\e}}+1}\]

  If $F^\e_2> F_1(x)$ then the above expression converges to $0$. This happens whenever $|\left<x,e_1\right>_H|> \e^{\kappa}$. The convergence is uniform for $x$ satisfying $\e^{\kappa/2}< |\left<x, e_1 \right>_H| \leq L $. On the other hand, if $|\left<x,e_1\right>_H|\leq \e^{\kappa/2}$, then $|u^\e(x)-u(x)|_H < \frac{2\alpha_1\e^\kappa}{\lambda_1}$, because $\rho_1^\delta$ is bounded. This implies uniform convergence on bounded subsets of $H$.
\end{proof}

\begin{Theorem}
\label{T:Scheme2} Assume that $\delta=2\e$ and that for $\kappa\in(0,1)$, $\e^{1-\kappa}\leq\frac{\alpha_{1}}{2\lambda_1^2}$. Let $u(t,x)=-B^{\star} D_{x}U^{\delta}(t,x)$ where $U^{\delta}$ is defined in (\ref{Eq:ControlDesignLinear2}). Then up to an
exponentially negligible term in $\e$, we have
\[
-\e\log \mathcal{Q}^{\e}(0,0;u^\e)\geq 
\frac{1}{2}(1-\eta)U^{\delta}(0,0)
\]
where
\[
U^{\delta}(0,0)\geq \frac{\alpha_{1}}{\lambda_1^2}(L^{2}-\e^{\kappa})  -\delta\log2.
\]
\end{Theorem}

\begin{proof}
Recalling (\ref{eq:Q-eps-lower-bound2}) we have
\begin{align*}
-\e \log(\mathcal{Q}^{\e}(0,x,u^\e)) &\geq\inf_{v \in \mathcal{A}} \Bigg( 2 Z(0,x) -2 \E Z(\hat{\tau}^\e, \hat{X}^\e(\hat{\tau}^\e)) \nonumber\\
&\qquad + 2 \E \int_0^{\hat{\tau}^\e}  \mathcal{G}^\e[Z,U](s,\hat{X}^{\e}(s)) ds \Bigg).
 \end{align*}

Choose $Z = U^{\delta,\eta}$ as in Lemma \ref{L:BoundGScheme2}. By Theorem \ref{thm:conv-of-endpoint} we have that $\lim_{\e\rightarrow 0}\E U^{\delta,\eta}(\hat{\tau}^\e, \hat{X}^\e(\hat{\tau}^\e))=0$. Therefore, we can find a small enough $\e>0$ such that this expression is less than $Z(0,x)/2$. By Lemma \ref{L:BoundGScheme2} we have that $\mathcal{G}^{\e}[Z,U](t,x)\geq 0$.  Since $U^{\delta}(x)$ is the exponential mollification of two functions, Lemma 4.1 of \cite{dsz-2014} gives that for every $x\in H$
\[
U^{\delta}(x)\geq \min\left\{F_{1}(x), F^{\e}_{2}\right\}-\delta\log 2,
\]
concluding the proof of the theorem.
\end{proof}

\section{Numerical simulations}\label{S:Numerics}
In this section we demonstrate the theoretical results of this paper by  a series of simulation studies for (\ref{eq:intro-X}).   Clearly, if the initial point is in the domain of attraction of the stable equilibrium point of the SPDE, then for $L>0$, we are dealing with a rare event. Hence accelerated Monte Carlo methods such as importance samplings become relevant. We will apply the schemes of Section \ref{S:IS_schemes} and we will compare their performance with (a): standard Monte Carlo, which corresponds to no-change of measure at all, and with various other alternatives such as (b) reversing the dynamics everywhere in the domain of simulation and (c): forcing all the directions in the region away from the rest point as opposed to the suggested change of measure, where only the important $e_{1}$ is being forced.

The mild solution to
\begin{equation}\label{Eq:SimulatedProcess}
dX^\e(t) = (AX^\e(t) +  Bu(t))dt + \sqrt{\e}Bd\bar{w}(t)
\end{equation}
is
\[X^\e(t) = e^{tA}x + \int_0^t e^{(t-s)A}Bu(s)ds + \sqrt{\e}\int_0^t e^{(t-s)A}Bd\bar{w}(s),\]
where $e^{tA}$ is the $C_{0}$-semigroup generated by $A$.

We will only use controls in feedback form, i.e, $u(s)=u(s,X^\e(s))$. It is clear that in order to simulate the process given by (\ref{Eq:SimulatedProcess}), we need to discretize the equation in time and space. Here one can use many different methods ranging from finite differences to spectral methods. In the simulation below we used the exponential Euler scheme finite-dimensional Galerkin projection as it is described in  \cite{jk-2009}. In particular, we first notice that the $N^{\text{th}}$ Galerkin approximation for $X^{\e}$ (\ref{Eq:SimulatedProcess}) is given by
\begin{equation}\label{Eq:GalerkinApprox_SimulatedProcess}
    \begin{cases}
      dX_N^\e(t) = (A_N X_N^\e(t) + (\Pi_N Bu)(t, X_N^\e(t)))dt + \sqrt{\e} \Pi_N Bd\bar{w}(t) \\
      X_N^\e(0) = \Pi_N x,
    \end{cases}
  \end{equation}
and we refer the reader to Appendix B for the definition of the projection operators $A_N$ and $\Pi_N$. Under appropriate conditions on the control $u(t,x)$, the unique solution to (\ref{Eq:GalerkinApprox_SimulatedProcess}) can be written as
\begin{align}\label{Eq:GalerkinApprox_SimulatedProcess-weak}
      X_N^\e(t) &= e^{A_N t}\Pi_N x+\int_{0}^{t}e^{A_N(t-s)}(\Pi_N Bu)(s, X_N^\e(s))ds \nonumber\\
      &\quad + \sqrt{\e} \int_{0}^{t}e^{A_N(t-s)}\Pi_N Bd\bar{w}(t)
\end{align}

The exponential Euler numerical scheme, as introduced in \cite{jk-2009}, that we use in order to simulate from (\ref{Eq:GalerkinApprox_SimulatedProcess-weak}) goes as follows. Consider time step $h=T/\Lambda$ for some $\Lambda\in\mathbb{N}$ and discretization times $t_{k}=kh$ for $k=0,\cdots,\Lambda$. Then we set $\Theta_{0}^{N,\Lambda}=\Pi_N x$ and we define
\begin{align}\label{Eq:ExponentialEulerScheme1}
      \Theta_{k+1}^{N,\Lambda} &= e^{A_N h}\Theta_{k}^{N,\Lambda}+A_{N}^{-1}(e^{A_{N}h}-I) (\Pi_N Bu)(t_{k}, \Theta_{k}^{N,\Lambda}) \nonumber\\
      &\qquad + \sqrt{\e} \int_{t_{k}}^{t_{k+1}}e^{A_N(t_{k+1}-s)}\Pi_N Bd\bar{w}(t)
\end{align}

In particular, for given $N$ and $\Lambda$, and for $ \Theta_{k,j}^{N,\Lambda}=\left<e_{j},  \Theta_{k}^{N,\Lambda}\right>$ and $f_{N}^{j}=\left<e_{j}, \Pi_N Bu(t_k,\Theta_k^{N,\Lambda})\right>$ with $j=1,\cdots, N$, we have that the numerical scheme for the approximation to (\ref{Eq:SimulatedProcess}) is
\begin{equation}\label{Eq:ExponentialEulerScheme2}
      \Theta_{k+1,j}^{N,\Lambda} = e^{-\alpha_{j} h}\Theta_{k,j}^{N,\Lambda}+\frac{1-e^{-\alpha_{j}h}}{\alpha_{j}} f_{N}^{j} + \sqrt{\e} \lambda_j\sqrt{\frac{1-e^{-2\alpha_{j}h}}{2\alpha_{j}}} \xi_{k}^{j}
\end{equation}
where $\xi_{k}^{j}$ for $k=0,\cdots, \Lambda-1$ and $j=1,\cdots, N$ are independent, standard normally distributed random variables. As it is quantified in Theorem 3.1 and more precisely in Section 4(b) of \cite{jk-2009}, the strong error rate of convergence of this scheme for our case of interest is $N^{-(1/2)+\zeta}+\frac{\log \Lambda}{\Lambda}$ for an arbitrarily small $\zeta>0$.

In our numerical simulations we consider the stochastic heat equation and we take $B=I$ (i.e., we work with space-time white noise) and the operator $A$ to be the  realization of the Laplace operator $\frac{\partial^2}{\partial \xi^2}$ with Dirichlet boundary conditions in $H$. $A$ is diagonalizable in $H$. The complete orthonormal basis of $H$  is given by
\begin{equation}
  e_k(\xi) = \sqrt{2} \sin(k\pi\xi), \ \ k=1,2,3,...
\end{equation}
with the eigenvalues taking the form $\alpha_{k}= k^2\pi^2$. Our goal is to estimate quantities of the form (\ref{eq:intro-theta}).

All the simulations below were done using a parallel MPI C code with  $K=5\times10^5$ Monte Carlo trajectories and we consider exit from the ball of size $L=1$ of a system exposed to space-time white noise ($B=I$). As it is standard in the related literature, the measure of performance is relative error per sample, defined as
\[
\mbox{relative error per sample} \doteq\sqrt{K}\frac{\mbox{standard deviation of the estimator}}%
{\mbox{expected value of the estimator}}.
\]

The smaller the relative error per sample is, the more efficient the algorithm is and the more accurate the estimator is. However, in practice both the standard deviation and the expected value
of an estimator are typically unknown, which implies that empirical relative error is
often used for measurement. This means that the expected value of
the estimator will be replaced by the empirical sample mean, and the
standard deviation of the estimator will be replaced by the
empirical sample standard error.

 Before presenting the simulation results, let us comment on what the end conclusions of the numerical studies are\footnote{ Due to space limitations issues and due to the lack of any important additional information, we do not report estimated probability values for some of the test cases and we only report estimated relative errors per sample, which is the measure of performance being used. The data on probability estimates is available upon request.}.
\begin{enumerate}
\item{Standard Monte Carlo estimation, i.e. with no change of measure performs pretty bad as it is indicated in Tables \ref{Table100aSM}-\ref{Table100bSM}. A dash line indicates that there was no successful trajectory in the simulations and thus no estimate, good or bad, could be provided. Notice that the relative errors per sample in Table \ref{Table100bSM} are getting increasingly large making the reported probability values of Table \ref{Table100aSM} to be of no value.}
\item{The importance sampling scheme based on Scheme 2 performs very well as it is indicated in Tables \ref{TableEstimatedValues}-\ref{Table300b} with increasing accuracy as $T$ gets larger. In Table  \ref{TableEstimatedValues} we have picked some representative probability estimates and we have compared the estimated values for different levels of  the $N^{th}$ Galerkin approximation with $N=4,100,150,300$. We notice that the estimates range from events of the order of $10^{-4}$ to $10^{-22}$ and that the estimates are practically indistinguishable for $N=100,150,300$. This indicates that the first mode really dominates the rare event. This also leads us to conclude that $N=100$ is a sufficiently good lower dimensional approximation to the corresponding SPDE. Notice also that the relative errors per sample as reported in Tables \ref{Table4b}-\ref{Table300b} support the theoretical findings in that the scheme performs optimally as the theory predicts. In particular as $T$ gets larger and $\e$ gets smaller, relative errors decrease independently of the dimension. }
\item{In Tables \ref{Table100b2}-\ref{Table100b22} we report estimated relative errors per sample based on Scheme 1. Comparing these tables with the ones corresponding to Scheme 2, i.e., Tables \ref{Table4b}-\ref{Table300b}, we notice that Scheme 1 seems to be performing  a little bit better than the simpler Scheme 2 for small times, but the difference in performance disappears as $T$ gets larger. We note however that the slightly superior performance comes with a little bit of extra computational cost, in that instead of $F^{\e}_{2}$ of Scheme 2, one needs to compute at each step both $F_{2}(t,x)$ and $F_{3}(t,x)$ of Scheme 1. }
\item{In Table \ref{Table100twodequal_b} we investigate numerically the situation where the first two eigenvalues are the same, the third eigenvalue is well separated from the first and second and we project down to the $\{e_{1},e_{2}\}$ manifold. We observe that the performance of the scheme (in terms of relative error per sample) is pretty stable as $T$ gets larger. At the same time, if the spectral gap exists but we still project down to the $\{e_{1},e_{2}\}$ manifold instead of the  $\{e_{1}\}$ manifold, then the performance is quite bad, see Table \ref{Table100twodNotequal_b}.}
 \item{Simulations based on forcing  the modes everywhere (either all of the modes or only the first one) were also implemented, see Table \ref{Table100NOmollification_b} for $N=100$. A clear degradation in performance is indicated as $T$ gets larger.}
 \item{If there is a spectral gap that is not sufficiently large, then the performance starts degrading, see Table \ref{Table100Mollification_NoSpectralGapb}. }
\end{enumerate}

\begin{table}[htbp!]
\begin{center}%
\begin{scriptsize}
\begin{tabular}
[c]{|c|c|c|c|c|c|c|c|c|}\hline
$\varepsilon\hspace{0.1cm} | \hspace{0.1cm} T$&$1$&$2$ &$3$&$4$ & $6$  & $8$ & $10$ & $12$ \\
\hline  $0.09$ & $-$ & $4.80e-05$   & $9.39e-04$  &$1.65e-04$  &$3.12e-04$ & $4.76e-04 $ & $6.32e-04 $ &$7.04e-04 $\\
\hline  $0.08$ & $-$ & $1.80e-05$   & $2.6e-05$   &$4.39e-05$  &$9.40e-05$ & $1.08e-04 $ & $1.44e-04 $ &$1.62e-04 $\\
\hline  $0.07$ & $-$ & $2.00e-06$   & $2.01e-06$  &$1.39e-05$  &$7.99e-06$ & $2.40e-05 $ & $2.41e-05 $ &$4.00e-05 $\\
\hline  $0.06$ & $-$ & $-$          & $-$         &$-$         &$1.99e-06$ & $-$         & $-$         &$6.00e-06 $\\
\hline  $0.05$ & $-$ & $-$          & $-$         &$-$         &$-$        & $-$         & $-$         &$-$\\
\hline  $0.04$ & $-$ & $-$          & $-$         &$-$         &$-$        & $-$         & $-$         &$-$\\
\hline  $0.03$ & $-$ & $-$          & $-$         &$-$         &$-$        & $-$         & $-$         &$- $\\
\hline  $0.02$ & $-$ & $-$          & $-$         &$-$         &$-$        & $-$         & $-$         &$-$\\
\hline
\end{tabular}
\end{scriptsize}
\end{center}
\caption{Estimated probability values for $\theta^\e(0,T)$ for different pairs $(\varepsilon,T)$ when Galerkin projection level is $N=100$. The values reported are based on standard Monte Carlo without employing some change of measure.  }
\label{Table100aSM}%
\end{table}

\begin{table}[htbp!]
\begin{center}%
\begin{small}
\begin{tabular}
[c]{|c|c|c|c|c|c|c|c|c|}\hline
$\varepsilon\hspace{0.1cm} | \hspace{0.1cm} T$&$1$&$2$ &$3$&$4$ & $6$  & $8$ & $10$ & $12$ \\
\hline  $0.09$ & $-$ & $144$   & $103$    &$78$  &$57$ & $46 $ & $40 $ &$38 $\\
\hline  $0.08$ & $-$ & $235$   & $196$    &$150$  &$103$ & $96 $ & $83 $ &$79 $\\
\hline  $0.07$ & $-$ & $707$   & $707$    &$267$  &$353$ & $204$ & $204$ &$158$\\
\hline  $0.06$ & $-$ & $-$   & $-$        &$-$  &$707$ & $-$ & $-$ &$408 $\\
\hline  $0.05$ & $-$ & $-$   & $-$        &$-$  &$-$ & $-$ & $-$ &$-$\\
\hline  $0.04$ & $-$ & $-$   & $-$        &$-$  &$-$ & $-$ & $-$ &$-$\\
\hline  $0.03$ & $-$ & $-$   & $-$        &$-$  &$-$ & $-$ & $-$ &$-$\\
\hline  $0.02$ & $-$ & $-$   & $-$        &$-$  &$-$ & $-$ & $-$ &$-$\\
\hline
\end{tabular}
\end{small}
\end{center}
\caption{Estimated relative errors per sample for $\theta^\e(0,T)$ for different pairs $(\varepsilon,T)$ when Galerkin projection level is $N=100$.  The values reported are based on standard Monte Carlo without employing some change of measure. Notice that a probability of $2\times 10^{-6}$ means that exactly one of the $5 \times 10^5$ trajectories exited the region. The relative error in that case is $707 = \sqrt{5\times 10^5}$ }
\label{Table100bSM}%
\end{table}

\begin{table}[htbp!]
\begin{center}%
\begin{scriptsize}
\begin{tabular}
[c]{|c|c|c|c|c|c|}\hline
$\varepsilon$ & $T$ &$N=4$&$N=100$ &$N=150$ & $N=300$  \\
\hline  $0.09$ & $2$ &  $4.08e-05$   & $4.47e-05$   & $4.48e-05$     &$4.55e-05$ \\
        $0.09$ & $4$ &  $1.58e-04$   & $1.75e-04$   & $1.76e-04$     &$1.74e-04$ \\
        $0.09$ & $8$ &  $3.99e-04$   & $4.42e-04$   & $4.44e-04$     &$4.43e-04$ \\
        $0.09$ & $12$&  $6.42e-04$   & $7.10e-04$   & $7.12e-04$     &$7.11e-04$ \\
\hline
        $0.06$ & $2$ & $1.37e-07 $   & $1.52e-07$   & $1.52e-07$     &$1.51e-07$ \\
        $0.06$ & $4$ & $6.75e-07 $   & $7.45e-07$   & $7.46e-07$     &$7.46e-07$ \\
        $0.06$ & $8$ & $1.81e-06 $   & $2.01e-06$   & $2.00e-06$     &$2.00e-06$ \\
        $0.06$ & $12$& $2.93e-06 $   & $3.26e-06$   & $3.26e-06$     &$3.27e-06$ \\
\hline
        $0.04$ & $2$ & $2.57e-11 $   & $2.91e-11$   & $2.92e-11$     &$2.92e-11$ \\
        $0.04$ & $4$ & $1.74e-10 $   & $1.93e-10$   & $1.94e-10$     &$1.94e-10$ \\
        $0.04$ & $8$ & $4.96e-10 $   & $5.56e-10$   & $5.53e-10$     &$5.53e-10$ \\
        $0.04$ & $12$& $8.26e-10 $   & $9.16e-10$   & $9.21e-10$     &$9.20e-10$ \\
\hline
        $0.02$ & $2$ & $1.87e-22 $   & $2.08e-22$   & $2.13e-22$     &$2.23e-22$ \\
        $0.02$ & $4$ & $2.63e-21 $   & $2.93e-21$   & $2.94e-21$     &$2.97e-21$ \\
        $0.02$ & $8$ & $8.67e-21 $   & $9.61e-21$   & $9.51e-21$     &$9.61e-21$ \\
        $0.02$ & $12$& $1.46e-20 $   & $1.61e-20$   & $1.63e-20$     &$1.62e-20$ \\
\hline
\end{tabular}
\end{scriptsize}
\end{center}
\caption{Estimated probability values for $\theta^\e(0,T)$ for different pairs $(\varepsilon,T)$ and for Galerkin projection levels of $N=4,100,150,300$.  The importance sampling scheme being used is Scheme 2 of Section \ref{S:IS_schemes} with $\kappa=0.6$. }
\label{TableEstimatedValues}%
\end{table}


\begin{table}[htbp!]
\begin{center}%
\begin{small}
\begin{tabular}
[c]{|c|c|c|c|c|c|c|c|c|}\hline
$\varepsilon\hspace{0.1cm} | \hspace{0.1cm} T$&$1$&$2$ &$3$&$4$ & $6$  & $8$ & $10$ & $12$ \\
\hline  $0.09$ & $12.2 $ & $3.3$   & $2.0$     &$1.6$  &$1.1$ & $0.9 $ & $0.9 $ &$0.9 $\\
\hline  $0.08$ & $14.9 $ & $3.6$   & $2.1$     &$1.7$  &$1.2$ & $1.0 $ & $0.9 $ &$0.9 $\\
\hline  $0.07$ & $18.3 $ & $3.9$   & $3.4$     &$1.8$  &$1.3$ & $1.1 $ & $0.9 $ &$0.9 $\\
\hline  $0.06$ & $25.9 $ & $4.4$   & $2.6$     &$1.9$  &$1.4$ & $1.1 $ & $1.0 $ &$0.9 $\\
\hline  $0.05$ & $39.5 $   & $5.2$    & $2.9$      &$2.2$  &$1.6$ & $1.3 $ & $1.1 $ &$1.0 $\\
\hline  $0.04$ & $70.8 $ & $6.6$   & $3.4$     &$2.5$  &$1.8$ & $1.4 $ & $1.2 $ &$1.1 $\\
\hline  $0.03$ & $145 $ & $8.9$   & $4.3$    &$3.1$  &$2.1$ & $1.7 $ & $1.4 $ &$1.3 $\\
\hline  $0.02$ & $- $   & $16.1$   & $6.3$    &$4.2$  &$2.9$ & $2.3 $ & $1.9 $ &$1.7 $\\
\hline
\end{tabular}
\end{small}
\end{center}
\caption{Estimated relative errors per sample for $\theta^\e(0,T)$ for different pairs $(\varepsilon,T)$ when Galerkin projection level is $N=4$.  The importance sampling scheme being used is Scheme 2 of Section \ref{S:IS_schemes} with $\kappa=0.6$. }
\label{Table4b}%
\end{table}


\begin{table}[htbp!]
\begin{center}%
\begin{small}
\begin{tabular}
[c]{|c|c|c|c|c|c|c|c|c|}\hline
$\varepsilon\hspace{0.1cm} | \hspace{0.1cm} T$&$1$&$2$ &$3$&$4$ & $6$  & $8$ & $10$ & $12$ \\
\hline  $0.09$ & $12 $ & $3.3$   & $2.1$     &$1.6$  &$1.1$ & $1.0 $ & $0.9 $ &$0.9 $\\
\hline  $0.08$ & $15 $ & $3.5$   & $2.2$     &$1.7$  &$1.2$ & $1.0 $ & $0.9 $ &$0.9 $\\
\hline  $0.07$ & $18 $ & $3.9$   & $2.4$     &$1.8$  &$1.3$ & $1.1 $ & $0.9 $ &$0.9 $\\
\hline  $0.06$ & $25 $ & $4.5$   & $2.6$     &$1.9$  &$1.4$ & $1.2 $ & $1.0 $ &$0.9 $\\
\hline  $0.05$ & $42 $ & $5.3$   & $2.9$     &$2.1$  &$1.5$ & $1.3 $ & $1.1 $ &$1.0 $\\
\hline  $0.04$ & $64 $ & $6.6$   & $3.4$     &$2.5$  &$1.8$ & $1.4 $ & $1.2 $ &$1.1 $\\
\hline  $0.03$ & $187$ & $9.0$   & $4.3$     &$3.0$  &$2.1$ & $1.7 $ & $1.5 $ &$1.3 $\\
\hline  $0.02$ & $-$  & $16.4$   & $6.2$     &$4.2$  &$2.9$ & $2.3 $ & $1.9 $ &$1.8 $\\
\hline
\end{tabular}
\end{small}
\end{center}
\caption{Estimated relative errors per sample for $\theta^\e(0,T)$ for different pairs $(\varepsilon,T)$ when Galerkin projection level is $N=100$.  The importance sampling scheme being used  is Scheme 2 of Section \ref{S:IS_schemes} with $\kappa=0.6$. }
\label{Table100b}%
\end{table}


\begin{table}[htbp!]
\begin{center}%
\begin{small}
\begin{tabular}
[c]{|c|c|c|c|c|c|c|c|c|}\hline
$\varepsilon\hspace{0.1cm} | \hspace{0.1cm} T$&$1$&$2$ &$3$&$4$ & $6$  & $8$ & $10$ & $12$ \\
\hline  $0.09$ & $11.9 $ & $3.3$   & $2.1$    &$1.6$  &$1.1$ & $1.0 $ & $0.9 $ &$0.9 $\\
\hline  $0.08$ & $14.4 $ & $3.5$   & $2.2$    &$1.7$  &$1.2$ & $1.0 $ & $0.9 $ &$0.9 $\\
\hline  $0.07$ & $18.8 $ & $3.9$   & $2.4$    &$1.9$  &$1.3$ & $1.1 $ & $0.9 $ &$0.9 $\\
\hline  $0.06$ & $24.7 $ & $4.5$   & $2.6$    &$2.1$  &$1.4$ & $1.2 $ & $1.0 $ &$0.9 $\\
\hline  $0.05$ & $38.5 $ & $5.3$   & $2.9$    &$2.5$  &$1.5$ & $1.2 $ & $1.1 $ &$1.0 $\\
\hline  $0.04$ & $66.1 $ & $6.6$   & $3.5$     &$2.7$  &$1.8$ & $1.4 $ & $1.2 $ &$1.1 $\\
\hline  $0.03$ & $189 $ & $9.1$   & $4.3$     &$3.0$  &$2.1$ & $1.7 $ & $1.5 $ &$1.3 $\\
\hline  $0.02$ & $-$ & $16.2$   & $6.3$     &$4.2$  &$2.9$ & $2.3 $ & $1.9 $ &$1.8 $\\
\hline
\end{tabular}
\end{small}
\end{center}
\caption{Estimated relative errors per sample for $\theta^\e(0,T)$ for different pairs $(\varepsilon,T)$ when Galerkin projection level is $N=150$.  The importance sampling scheme being used is Scheme 2 of Section \ref{S:IS_schemes} with $\kappa=0.6$. }
\label{Table150b}%
\end{table}

%

\begin{table}[htbp!]
\begin{center}%
\begin{small}
\begin{tabular}
[c]{|c|c|c|c|c|c|c|c|c|}\hline
$\varepsilon\hspace{0.1cm} | \hspace{0.1cm} T$&$1$&$2$ &$3$&$4$ & $6$  & $8$ & $10$ & $12$ \\
\hline  $0.09$ & $11.9 $ & $3.2$   & $2.1$    &$1.6$  &$1.1$ & $1.0 $ & $0.9 $ &$0.9 $\\
\hline  $0.08$ & $14.2 $ & $3.6$   & $2.2$    &$1.7$  &$1.2$ & $1.0 $ & $0.9 $ &$0.9 $\\
\hline  $0.07$ & $18.3 $ & $3.9$   & $2.3$    &$1.8$  &$1.3$ & $1.1 $ & $0.9 $ &$0.9 $\\
\hline  $0.06$ & $25.3 $ & $4.5$   & $2.6$    &$1.9$  &$1.4$ & $1.2 $ & $1.0 $ &$0.9 $\\
\hline  $0.05$ & $39.3 $ & $5.3$   & $2.9$    &$2.1$  &$1.5$ & $1.2 $ & $1.1 $ &$1.0 $\\
\hline  $0.04$ & $68.6 $ & $6.5$   & $3.4$     &$2.5$  &$1.8$ & $1.4 $ & $1.2 $ &$1.1 $\\
\hline  $0.03$ & $161 $ & $9.2$   & $4.3$     &$3.0$  &$2.1$ & $1.7 $ & $1.5 $ &$1.3 $\\
\hline  $0.02$ & $500$ & $15.7$   & $6.3$     &$4.2$  &$2.9$ & $2.3 $ & $1.9 $ &$1.7 $\\
\hline
\end{tabular}
\end{small}
\end{center}
\caption{Estimated relative errors per sample for $\theta^\e(0,T)$ for different pairs $(\varepsilon,T)$ when Galerkin projection level is $N=300$.  The importance sampling scheme being used is Scheme 2 of Section \ref{S:IS_schemes} with $\kappa=0.6$. }
\label{Table300b}%
\end{table}

\begin{table}[htbp!]
\begin{center}%
\begin{small}
\begin{tabular}
[c]{|c|c|c|c|c|c|c|c|c|}\hline
$\varepsilon\hspace{0.1cm} | \hspace{0.1cm} T$&$1$&$2$ &$3$&$4$ & $6$  & $8$ & $10$ & $12$ \\
\hline  $0.09$ & $3.7 $ & $1.1$   & $1.1$    &$1.4$  &$1.3$ & $1.2 $ & $1.1 $ &$1.0 $\\
\hline  $0.08$ & $4.3 $ & $1.2$   & $1.0$    &$1.4$  &$1.4$ & $1.3 $ & $1.1 $ &$1.0 $\\
\hline  $0.07$ & $5.2 $ & $1.3$   & $0.9$    &$1.5$  &$1.6$ & $1.4 $ & $1.3 $ &$1.2 $\\
\hline  $0.06$ & $6.4 $ & $1.4$   & $0.9$    &$1.5$  &$1.7$ & $1.6 $ & $1.4 $ &$1.3 $\\
\hline  $0.05$ & $8.7 $ & $1.6$   & $0.9$    &$1.5$  &$1.9$ & $1.8 $ & $1.7 $ &$1.6 $\\
\hline  $0.04$ & $13.8 $ & $1.6$   & $0.9$     &$1.2$  &$2.3$ & $2.3 $ & $2.1 $ &$1.9 $\\
\hline  $0.03$ & $29.1 $ & $1.8$   & $1.0$     &$1.1$  &$3.1$ & $3.1 $ & $2.9 $ &$2.7 $\\
\hline  $0.02$ & $121.9 $ & $2.4$   & $1.1$     &$1.0$  &$4.6$ & $5.2 $ & $5.1 $ &$4.7 $\\
\hline
\end{tabular}
\end{small}
\end{center}
\caption{Estimated relative errors per sample for $\theta^\e(0,T)$ for different pairs $(\varepsilon,T)$ when Galerkin projection level is $N=100$.  The importance sampling scheme being used is  Scheme 1 of Section \ref{S:IS_schemes}  with parameters $(M,t^{*})=(\e^{-0.5}, -\frac{2}{\alpha_{1}}\log(\e^{0.5}))$. }
\label{Table100b2}%
\end{table}

\begin{table}[htbp!]
\begin{center}%
\begin{small}
\begin{tabular}
[c]{|c|c|c|c|c|c|c|c|c|}\hline
$\varepsilon\hspace{0.1cm} | \hspace{0.1cm} T$&$1$&$2$ &$3$&$4$ & $6$  & $8$ & $10$ & $12$ \\
\hline  $0.09$ & $3.7 $ & $1.5$   & $1.6$    &$1.4$  &$1.1$ & $0.9 $ & $0.8 $ &$0.8 $\\
\hline  $0.08$ & $4.3 $ & $1.5$   & $1.6$    &$1.4$  &$1.2$ & $1.0 $ & $0.9 $ &$0.8 $\\
\hline  $0.07$ & $5.2 $ & $1.3$   & $1.6$    &$1.5$  &$1.2$ & $1.1 $ & $0.9 $ &$0.9 $\\
\hline  $0.06$ & $6.4 $ & $1.4$   & $1.6$    &$1.6$  &$1.3$ & $1.2 $ & $1.0 $ &$0.9 $\\
\hline  $0.05$ & $8.7 $ & $1.6$   & $1.4$    &$1.7$  &$1.5$ & $1.3 $ & $1.1 $ &$1.1 $\\
\hline  $0.04$ & $13.8 $ & $1.6$   & $1.2$     &$1.7$  &$1.6$ & $1.4 $ & $1.2 $ &$1.2 $\\
\hline  $0.03$ & $29.1 $ & $1.8$   & $1.0$     &$1.7$  &$1.9$ & $1.7 $ & $1.6 $ &$1.4 $\\
\hline  $0.02$ & $118.8 $ & $2.4$   & $1.0$     &$1.4$  &$2.4$ & $2.3 $ & $2.1 $ &$1.9 $\\
\hline
\end{tabular}
\end{small}
\end{center}
\caption{Estimated relative errors per sample for $\theta^\e(0,T)$ for different pairs $(\varepsilon,T)$ when Galerkin projection level is $N=100$.  The importance sampling scheme being used is  Scheme 1 of Section \ref{S:IS_schemes}  with parameters $(M,t^{*})=(\e^{-0.6}, -\frac{2}{\alpha_{1}}\log(\e^{0.6}))$. }
\label{Table100b22}%
\end{table}

\begin{table}[htbp!]
\begin{center}%
\begin{small}
\begin{tabular}
[c]{|c|c|c|c|c|c|c|c|c|}\hline
$\varepsilon\hspace{0.1cm} | \hspace{0.1cm} T$&$1$&$2$ &$3$&$4$ & $6$  & $8$ & $10$ & $12$ \\
\hline  $0.09$ & $5.8 $ & $1.7$   & $1.1$    &$0.9$  &$0.9$ & $1.0 $ & $1.4 $ &$1.8 $\\
\hline  $0.08$ & $6.9 $ & $1.8$   & $1.2$    &$0.9$  &$0.9$ & $1.0 $ & $1.3 $ &$1.7 $\\
\hline  $0.07$ & $8.6 $ & $2.1$   & $1.3$    &$1.0$  &$0.9$ & $1.0 $ & $1.2 $ &$1.6 $\\
\hline  $0.06$ & $11.8 $ & $2.3$   & $1.4$    &$1.0$  &$0.9$ & $1.0 $ & $1.1 $ &$1.4 $\\
\hline  $0.05$ & $17.1 $ & $2.7$   & $1.6$    &$1.1$  &$0.9$ & $1.0 $ & $1.0 $ &$1.2 $\\
\hline  $0.04$ & $28.8 $ & $3.3$   & $1.8$     &$1.3$  &$1.0$ & $0.9 $ & $1.0 $ &$1.1 $\\
\hline  $0.03$ & $76.5 $ & $4.5$   & $2.2$     &$1.6$  &$1.1$ & $0.9 $ & $1.0 $ &$1.0 $\\
\hline  $0.02$ & $500$ & $7.4$   & $3.2$     &$2.2$  &$1.5$ & $1.2 $ & $1.1 $ &$1.0 $\\
\hline
\end{tabular}
\end{small}
\end{center}
\caption{Estimated relative errors per sample for $\theta^\e(0,T)$ for different pairs $(\varepsilon,T)$ when Galerkin projection level is $N=100$. We project to the $\{e_{1},e_{2}\}$ manifold and the eigenvalues are $\alpha_{1}=\alpha_{2}=1$ while $\alpha_{k}=k^{2}$ for $k\geq 3$. The importance sampling scheme being used is Scheme 2 of Section \ref{S:IS_schemes} with $\kappa=0.6$. }
\label{Table100twodequal_b}%
\end{table}

\begin{table}[htbp!]
\begin{center}%
\begin{small}
\begin{tabular}
[c]{|c|c|c|c|c|c|c|c|c|}\hline
$\varepsilon\hspace{0.1cm} | \hspace{0.1cm} T$&$1$&$2$ &$3$&$4$ & $6$  & $8$ & $10$ & $12$ \\
\hline  $0.09$ & $175 $ & $54$   & $34$    &$26$  &$20$ & $16 $ & $15 $ &$13 $\\
\hline  $0.08$ & $220 $ & $82$   & $47$    &$38$  &$29$ & $23 $ & $21 $ &$19 $\\
\hline  $0.07$ & $360 $ & $118$   & $83$    &$63$  &$45$ & $37 $ & $32 $ &$29 $\\
\hline  $0.06$ & $453 $ & $250$   & $132$    &$104$  &$72$ & $59 $ & $49 $ &$48 $\\
\hline  $0.05$ & $471 $ & $341$   & $220$    &$168$  &$146$ & $101 $ & $83 $ &$85 $\\
\hline  $0.04$ & $543 $ & $499$   & $483$     &$302$  &$297$ & $268 $ & $222 $ &$155 $\\
\hline  $0.03$ & $- $ & $500$   & $500$     &$453$  &$402$ & $364 $ & $353$ &$397 $\\
\hline  $0.02$ & $-$ & $-$   & $-$     &$-$  &$-$ & $- $ & $- $ &$- $\\
\hline
\end{tabular}
\end{small}
\end{center}
\caption{Estimated relative errors per sample for $\theta^\e(0,T)$ for different pairs $(\varepsilon,T)$ when Galerkin projection level is $N=100$. We project to the $\{e_{1},e_{2}\}$ manifold and the eigenvalues are $k^{2}$ for $k\geq 1$. The importance sampling scheme being used is Scheme 2 of Section \ref{S:IS_schemes} with $\kappa=0.6$. }
\label{Table100twodNotequal_b}%
\end{table}

\begin{table}[htbp!]
\begin{center}%
\begin{small}
\begin{tabular}
[c]{|c|c|c|c|c|c|c|c|c|}\hline
$\varepsilon\hspace{0.1cm} | \hspace{0.1cm} T$&$1$&$2$ &$3$&$4$ & $6$  & $8$ & $10$ & $12$ \\
\hline  $0.09$ & $3.7 $ & $1.1$   & $0.9$    &$1.1$  &$2.2$ & $4.5 $ & $10 $ &$21 $\\
\hline  $0.08$ & $4.4 $ & $1.2$   & $0.9$    &$1.1$  &$2.2$ & $4.5 $ & $10 $ &$20 $\\
\hline  $0.07$ & $5.1 $ & $1.2$   & $0.9$    &$1.1$  &$2.1$ & $4.2 $ & $10 $ &$17 $\\
\hline  $0.06$ & $6.4 $ & $1.3$   & $0.9$    &$1.1$  &$2.2$ & $4.2 $ & $9.5 $ &$17 $\\
\hline  $0.05$ & $8.7 $ & $1.4$   & $0.9$    &$1.0$  &$2.1$ & $3.9 $ & $8.5 $ &$17 $\\
\hline  $0.04$ & $13.6 $ & $1.6$   & $0.9$     &$1.0$  &$1.9$ & $3.9 $ & $8.9 $ &$14 $\\
\hline  $0.03$ & $28.2 $ & $1.8$   & $0.9$     &$1.0$  &$1.8$ & $3.6$ & $7.6$ &$15 $\\
\hline  $0.02$ & $128.3$ & $2.4$   & $1.1$          &$0.9$  &$1.7$ & $3.4 $ & $7.4 $ &$16 $\\
\hline
\end{tabular}
\end{small}
\end{center}
\caption{Estimated relative errors per sample for $\theta^\e(0,T)$ for different pairs $(\varepsilon,T)$ when Galerkin projection level is $N=100$. We project to the $\{e_{1}\}$ manifold everywhere (i.e., even in the neighborhood of the attractor) and the eigenvalues are $k^{2}$ for $k\geq 1$. }
\label{Table100NOmollification_b}%
\end{table}

\begin{table}[htbp!]
\begin{center}%
\begin{small}
\begin{tabular}
[c]{|c|c|c|c|c|c|c|c|c|}\hline
$\varepsilon\hspace{0.1cm} | \hspace{0.1cm} T$&$1$&$2$ &$3$&$4$ & $6$  & $8$ & $10$ & $12$ \\
\hline  $0.09$ & $14.8 $ & $5.6$   & $2.9$    &$2.3$  &$2.2$ & $1.8 $ & $1.5 $ &$2.0 $\\
\hline  $0.08$ & $21.9 $ & $5.7$   & $3.0$    &$2.5$  &$2.3$ & $1.9 $ & $3.7 $ &$3.6 $\\
\hline  $0.07$ & $20.9 $ & $5.3$   & $3.2$    &$2.5$  &$2.1$ & $1.7 $ & $1.6 $ &$2.1 $\\
\hline  $0.06$ & $38.6 $ & $7.7$   & $4.1$    &$2.8$  &$13.6$ & $2.0 $ & $1.9 $ &$2.7 $\\
\hline  $0.05$ & $56.4 $ & $7.4$   & $4.1$    &$3.1$  &$2.2$ & $1.9 $ & $1.8 $ &$2.8 $\\
\hline  $0.04$ & $73.4 $ & $8.6$  & $6.1$    &$3.8$  &$2.6$ & $2.2 $ & $2.4 $ &$2.7 $\\
\hline  $0.03$ & $258.7 $ & $10.6$  & $5.1$    &$4.1$  &$2.8$ & $2.4$  & $2.2$  &$2.8 $\\
\hline  $0.02$ & $-$ & $17.7$  & $9.0$    &$5.4$  &$4.1$ & $3.6 $ & $2.8 $ &$3.5 $\\
\hline
\end{tabular}
\end{small}
\end{center}
\caption{Estimated relative errors per sample for $\theta^\e(0,T)$ for different pairs $(\varepsilon,T)$ when Galerkin projection level is $N=100$. We project to the $\{e_{1}\}$ manifold and the eigenvalues are $1,2$ and then $k^{2}$ for $k\geq 3$. The importance sampling scheme being used is Scheme 2 of Section \ref{S:IS_schemes} with $\kappa=0.6$. }
\label{Table100Mollification_NoSpectralGapb}%
\end{table}

\section{Conclusions and future work}\label{S:Conclusions}
In this paper we studied the issues that arise in the design of importance sampling schemes for small noise infinite dimensional stochastic dynamical systems. We concentrated on the linear case where we could also provide conditions on the spectral gap and we could design importance sampling methods whose performance does not degrade due to the infinite dimensionality of the system or due to prelimit effects.

Our results are a promising first step towards building provably efficient and implementable importance samplings schemes for the estimation of rare events for nonlinear SPDEs. There are however certain issues that need to be first understood better before addressing the nonlinear case. In the linear case the idea is that if a sufficiently large spectral gap exists then the rare event takes place in a lower dimensional manifold which happens to be affine (the $e_1$ direction). In this case we built our importance sampling schemes by projecting to the span of $e_1$. In Section \ref{S:IS} we demonstrated that projecting onto lower-dimensional manifolds is necessary to avoid infinite trace second derivatives. In the nonlinear case, the first problem is identifying the lower dimensional manifold where the rare event takes place (if one exists) and the second problem is how to project to this manifold if it is not affine. Solving the first problem seems to be problem dependent and it may be difficult to find closed form solutions. Even when one can explicitly characterize the lower-dimensional manifold where the rare event is likely to take place, it is not clear how to project onto that manifold if it is not linear. The results of this paper could potentially be useful for building sensible algorithms where one could  linearize locally the underlying manifold where the rare event takes place and apply the methods of this paper in a local fashion. These issues present interesting challenges for future work.

\appendix
\section{Proof of Lemma \ref{L:BoundGScheme2}}\label{A:ProofLemmaBoundGScheme2}

Before proving Lemma \ref{L:BoundGScheme2} let us define some useful quantities. Set
\begin{align*}
\beta_{0}(x)  &  =\left[  \rho_{1}(x)\left| B^{\star} D_{x}F_{1}(x)\right|^{2}_{H}-\rho_{1}^{2}(x)\left| B^{\star}  D_{x}F_{1}(x)\right|^{2}_{H}\right]
\end{align*}
and notice that $\rho_{1}\in[0,1]$, guarantees that $\beta_{0}(x)\geq0$. In addition, let us define
\[
\gamma_{1}=\mathcal{G}^{\e}[F_{1}](x)=-\e \alpha_{1}.
\]
By the argument of Lemma 4.1 of \cite{dsz-2014}
applied  to $\mathcal{G}^{\varepsilon}[U^{\delta,\eta}](x)$
and (\ref{Eq:Goperator}) we get%
\begin{align}
\mathcal{G}^{\e}[U^{\delta,\eta},U^{\delta}](x)\geq &
\frac{1-\eta}{2}\left(  1-\frac{\e}{\delta}\right)\beta_{0}(x)+(1-\eta)\rho_{1}(x)\gamma_{1} \nonumber\\ &\qquad+\frac{\eta-2\eta^{2}}{2}\rho_{1}^{2}(x) \left|B^{\star} D_{x}F_{1}(x)\right|^{2}_{H}
\label{Eq:bound_for_moll}%
\end{align}
for all $x\in H$. The lower bound for the operator $\mathcal{G}^{\e}[U^{\delta,\eta},U^{\delta}](x)$, given by (\ref{Eq:bound_for_moll}), will be based on a separate analysis for  three
different regions that are determined by level sets of $V_{1}(x)=|<x,e_{1}>|^{2}$.

Let $\kappa\in(0,1)$ to be chosen, $\alpha\in(0,1-\kappa)$ and consider $K$ such that $\frac{e^{-K}}{e^{-K}+1}=\frac{3}{4}$, i.e., $K=-\ln 3<0$. Let us also assume $\e\in(0,1)$ . Then, we define
\begin{eqnarray}
B_{1}&=&\left\{x\in H: V_{1}(x)\leq \e^{\kappa+\alpha}, \kappa\in(0,1),\alpha\in(0,1-\kappa) \right\}\nonumber\\
B_{2}&=&\left\{x\in H: \e^{\kappa+\alpha}\leq V_{1}(x)\leq \e^{\kappa}+ \left(\e^{\kappa}-\e K\right)\right\}\nonumber\\
B_{3}&=&\left\{x\in H:  \e^{\kappa}+ \left(\e^{\kappa}-\e K\right) \leq V_{1}(x)\leq L^{2}\right\}\nonumber
\end{eqnarray}

Lemma \ref{L:BoundGScheme2} is a direct consequence of Lemmas \ref{L:Region_0_y}, \ref{L:Region_hatx_A} and \ref{L:Region_y_hatx} that treat the regions $B_{1}, B_{3}$ and $B_{2}$ respectively.
\begin{Lemma}
\label{L:Region_0_y} Assume that $x\in B_{1}$, $\delta=2\e$, $\eta\leq1/2$ and $\e\in(0,1)$. Then, up to an
exponentially negligible term
\[
\mathcal{G}^{\varepsilon}[U^{\delta,\eta},U^{\delta}](x)\geq0.
\]
\end{Lemma}

\begin{proof}
In this region, we are guaranteed that $F_{1}(x)> F_{2}^{\e}$. Indeed, we have that
\begin{align}
F_{1}(x)- F_{2}^{\e}&\geq \frac{\alpha_{1}}{\lambda_1^2}\left(\e^{\kappa}-\e^{\kappa+\alpha}\right)> 0\nonumber
\end{align}
since $\e<1$ and $\alpha\in(0,1)$. Hence, we have that
\begin{equation}
-\frac{1}{2\e}\left[F_{1}(x)- F_{2}\right]\leq -\frac{1}{2\e}\left[\e^{\kappa}\left(1-\e^{\alpha}\right)\right]=-\frac{1-\e^{\alpha}}{2\e^{1-\kappa}}\nonumber
\end{equation}

This immediately implies that the term
involving the weight $\rho_{1}$ is exponentially
negligible. Since $\beta_{0}(x)\geq0$ and $\eta
\leq1/2$, all other terms are non-negative, and the result follows.
\end{proof}

\begin{Lemma}
\label{L:Region_hatx_A} Assume that $x\in B_{3}$, $\delta=2\e$, $\eta\leq1/4$ and that $\e^{1-\kappa}\in(0,\alpha_{1}/2\lambda_1^2)$.  Then, we have%
\begin{equation*}
\mathcal{G}^{\e}[U^{\delta,\eta},U^{\delta}](x)\geq0.
\label{Eq:Bound2_b}%
\end{equation*}
\end{Lemma}

\begin{proof}
In this region we have that $V(x)\geq 2\e^{\kappa}-\e \beta K>0$ for $\e$ small enough. Moreover, since $K=-\ln 3$ is chosen such that
\[
\frac{\epsilon^{-K}}{\epsilon^{-K}+1}=\frac{3}{4}
\]
we obtain that for $x\in B_{3}$,  $\rho_{1}(x)\geq 3/4$. We have the following inequalities
\begin{align*}
\mathcal{G}^{\e}[U^{\delta,\eta},U^{\delta}](x)  &  \geq (1-\eta)\frac{1}{4}\rho_{1}(x)(1-\rho_{1}(x))\left|B^{\star} D_{x} F_{1}(x)\right|^{2}_{H}+(1-\eta)\rho_{1}(x)\gamma_{1}\\
&\qquad+\frac{1}{2}\left(  \eta-2\eta^{2}\right)\left| \rho_{1}(x)B^{\star} D_{x}F_{1}(x)\right|^{2}\\
&\geq (1-\eta)\left[\rho_{1}(x)(1-\rho_{1}(x))\frac{\alpha_{1}^{2}}{\lambda_1^2}V_{1}(x) -\e\alpha_{1} \rho_{1}(x)\right]\\
&\qquad + 2\left(  \eta-2\eta^{2}\right)\rho_{1}^{2}(x) \frac{\alpha_{1}^{2}}{\lambda_1^2}V_{1}(x)\\
&\geq (1-\eta)\left[\frac{\alpha^{2}_{1}}{4\lambda_1^2}V_{1}(x) -\e\alpha_{1}\right]\rho_{1}(x)+ \frac{3\frac{\alpha^{2}_{1}}{\lambda_1^2}}{2}\left(  \eta-2\eta^{2}\right)\rho_{1}(x)V_{1}(x)\\
&\geq (1-\eta)\left[\frac{\alpha^{2}_{1}}{4\lambda_1^2}\left(2\e^{\kappa}-\e K\right) -\e\alpha_{1}\right]\rho_{1}(x)+ \frac{9\alpha^{2}_{1}}{16\lambda_1^2}\eta\left(2\e^{\kappa}-\e K\right)\\
&\geq (1-\eta)\alpha_{1}\left[\frac{\alpha_{1}}{2\lambda_1^2}\e^{\kappa} -\e\right]\rho_{1}(x)+ \frac{9\alpha^{2}_{1}}{8\lambda_1^2}\eta \e^{\kappa}\\
&\geq 0
\end{align*}

In the third inequality we used the fact that $\rho_{1}(x)\geq3/4$ for $x\in B_{3}$. In the next inequality we used that $\eta\leq1/4$ and that for $x\in B_{3}$, $V_{1}(x)\geq 2\e^{\kappa} -\e K$. Lastly, in the last inequality, we used that $K<0$ and that $0<\e^{1-\kappa}<\alpha_{1}/(2\lambda_{1}^{2})$. This concludes the proof of the lemma.
\end{proof}

\begin{Lemma}
\label{L:Region_y_hatx} Assume that $x\in B_{2}$, $\eta\leq1/4$ and set $\delta=2\e$. Let $\e>0$ be small enough such that $\e^{1-\kappa}\leq \frac{\alpha_{1}}{2\lambda_{1}^{2}}$. Then we have that
\begin{align}
\mathcal{G}^{\e}[U^{\delta,\eta},U^{\delta}](x)\geq & 0\nonumber
\end{align}
\end{Lemma}

\begin{proof}
This is the most problematic region, since one cannot  guarantee that  $\rho_{1}$ is exponentially negligible or of order one.  We distinguish two cases depending on whether
$\rho_{1}(x)>1/2$ or $\rho_{1}(x)\leq1/2$.\medskip

For the case $\rho_{1}(t,x)>1/2$, one can just follow the proof of  Lemma \ref{L:Region_hatx_A}. Then,  one immediately gets that $\mathcal{G}^{\e}[U^{\delta,\eta},U^{\delta}](x)\geq0$.

Let us now study the case $\rho_{1}(x)\leq1/2$. Here we need to rely on the positive contribution of $\beta
_{0}(x)$. Dropping other terms on the right that are not possibly negative,
we obtain from (\ref{Eq:bound_for_moll}) that
\begin{align*}
\mathcal{G}^{\e}[U^{\delta,\eta},U^{\delta}](x)  &  \geq (1-\eta)\frac{\beta}{4}\rho_{1}(x)(1-\rho_{1}(x))\left|B^{\star} D_{x} F_{1}(x)\right|^{2}_{H}+(1-\eta)\rho_{1}(x)\gamma_{1}\nonumber\\
&\geq (1-\eta)\frac{\beta}{8}\rho_{1}(x)\left| B^{\star} D_{x} F_{1}(x)\right|^{2}_{H}+(1-\eta)\rho_{1}(x)\gamma_{1}
\end{align*}
where we used $\rho_{1}(t,x)\leq1/2$. Recalling now the definitions of $D_{x}F_{1}(x)$ and $\gamma_{1}$, we subsequently obtain
\begin{align}
\mathcal{G}^{\e}[U^{\delta,\eta},U^{\delta}](t,x)&\geq
(1-\eta)\left[\frac{\alpha_{1}^{2}}{2\lambda_{1}^{2}}V(x)-\e\alpha_{1}\right]\rho_{1}(x)\nonumber\\
&\geq
(1-\eta)\alpha_{1}\left[\frac{\alpha_{1}}{2\lambda_{1}^{2}}\e^{\kappa}-\e\right]\rho_{1}(x)\geq 0\nonumber
\label{Eq:BoundForGreyAreaCaseII}
\end{align}
In the last inequality we used that for $x\in B_{2}$ $V(x)\geq\e^{\kappa}$ and that $\e>0$ is small enough such that $\e^{1-\kappa}\leq \frac{\alpha_{1}}{2\lambda_{1}^{2}}$.
This concludes the proof of the lemma.
\end{proof}

\section{Galerkin approximation}
The goal of this section is to get an explicit bound in terms of $\e,T$ and the eigenvalues of the difference of $X^\e(t)$ and its finite dimensional Galerkin approximation, see also  \cite{jk-2009} for general bounds. Our goal is not to present the most general result possible, but rather to point out the issues related for the problem studied in thus paper in the simplest situation possible.
For $N \in \mathbb{N}$ let $H_N$ be the finite dimensional space $\sspan\{e_k\}_{k=1}^N$. Let $\Pi_N: H \to H_N$ be the projection operator onto this space. That is, for any $x \in H$,
\[\Pi_N x = \sum_{k=1}^N \left<x,e_k\right>_H e_k.\]
\begin{Definition}
 Letting $A_N := \Pi_N A$, the $N^{\textrm{th}}$ Galerkin approximation for $X^\e$ is defined to be the solution to $N$-dimensional SDE
  \begin{equation}
    \begin{cases}
      dX_N^\e(t) = (A_N X_N^\e(t) + \Pi_N B u(X_N^\e(t)) )dt + \sqrt{\e} \Pi_N B d w(t) \\
      X_N^\e(0) = \Pi_N x.
    \end{cases}
  \end{equation}
\end{Definition}

Given that in this paper, $u$ represents the control being applied which turns out to be affine, we may, for the purposes of this section, embed this into A. We will do so and thus from now on set $u=0$. The same conclusions hold when $u\neq 0$.
\begin{Theorem} \label{thm:Galerkin}
  For any initial condition $x \in H$ and any $\e>0$, $T>0$,
  \begin{equation}
     \E \sup_{t \leq T}|X^\e(t) - X^\e_N(t)|_H^2  \leq  |(I-\Pi_N)x|_H + \sqrt{\e} CT\left(\sum_{k=N+1}^\infty \lambda_{k}^{2}\alpha_k^{-\gamma}  \right)^{1/2},
  \end{equation}
   for some constant $C<\infty$.   The limit as $N \to +\infty$ is zero but it is not uniform with respect to initial conditions in bounded subsets of $H$. The limit is uniform with respect to initial condition in the compact set $\{x\in H: |(-A)^\eta x|_H \leq R\}$ for any $\eta> 0$.
\end{Theorem}
Before proving this theorem in generality, we study the special case of the stochastic convolution.
\begin{Lemma} \label{lem:stoch-conv-Galerkin}
  For any $T>0$, $p\geq 1$, $\frac{1}{2}<\gamma<1$ the Galerkin approximations of the stochastic convolution converge in $L^p(\Omega;C([0,T];H))$ and there exists a constant $C=C(p,\gamma)$ such that
  \begin{equation}
    \E \sup_{t \leq T} \left|(I-\Pi_N)\int_0^t e^{(t-s)A}Bdw(s) \right|_H^p \leq C T \left(\sum_{k=N+1}^\infty \lambda_{k}^{2}\alpha_k^{-\gamma} \right)^{p/2}.
  \end{equation}
\end{Lemma}
\begin{proof}
  We use the stochastic factorization method (see \cite{DaP-Z}) which is based on the following identity. For any $s<t$, $0<\alpha<1$
  \begin{equation}
    \int_s^t (t-\sigma)^{\alpha - 1} (\sigma -s)^{-\alpha} d\sigma = \frac{\pi}{\sin(\alpha \pi)}.
  \end{equation}
  We then write the stochastic convolution as
  \begin{equation} \label{eq:stoch-factor}
    \int_0^t e^{(t-s)A}B d w(s) = \frac{\sin(\alpha\pi)}{\pi}\int_0^t (t-\sigma)^{\alpha -1} e^{(t-\sigma)A} Y_\alpha(\sigma)d\sigma
  \end{equation}
  \begin{equation}
    Y_\alpha(\sigma) = \int_0^\sigma (\sigma-s)^{-\alpha}e^{(\sigma-s)A} B d w(s).
  \end{equation}
  Let $\frac{1}{2}<\gamma<1$ and $p\geq 1$. We then choose $0<\alpha< \frac{1-\gamma}{2}$ and calculate that
  \[
  \E \left|(I-\Pi_N) Y_\alpha(\sigma) \right|_H^2 = \int_0^\sigma s^{-2\alpha} \sum_{k=N+1}^\infty \lambda_{k}^{2} e^{-2\alpha_k s} ds
  \]
  We use the identity $\sup_{x>0} x^\gamma e^{-x} =: C_\gamma <+\infty$ to show that
  \[e^{-2\alpha_k s} \leq \frac{e^{-\alpha_1 s}}{s^\gamma \alpha_k^\gamma}\]
  and it follows that there exists $C=C(\alpha,\gamma)$ such that
  \[\E |(I-\Pi_N)Y_\alpha(\sigma)|_H^2 \leq \left( \int_0^\infty s^{-2\alpha-\gamma} e^{-\alpha_1 s} ds\right) \sum_{k=N+1}^\infty \lambda_{k}^{2}\alpha_k^{-\gamma} \leq C\sum_{k=N+1}^\infty \lambda_{k}^{2}\alpha_k^{-\gamma}.  \]

  By the Burkholder-Davis-Gundy inequality, for any $p\geq 2$,
  \[\E |(I-\Pi_N) Y_\alpha(\sigma)|_H^p \leq C \left( \sum_{k=N+1}^\infty \lambda_{k}^{2}\alpha_k^{-\gamma} \right)^{p/2}.\]
  By applying the H\"older inequality to \eqref{eq:stoch-factor}
  \[\left|(I-\Pi_N) \int_0^t e^{(t-s)A}dw(s)\right|_H^p \leq \left(\int_0^t (t-\sigma)^{\frac{p(\alpha-1)}{p-1}} e^{-\frac{p\alpha_1(t-\sigma)}{p-1}} d\sigma \right)^{p-1} \left(\int_0^t |Y_\alpha(\sigma)|_H^p d\sigma \right).   \]
  If we choose $p$ large enough so that $\frac{p(\alpha-1)}{p}>-1$, then the first integral converges and is bouned for all $t>0$ and
  \[ \E \sup_{t \leq T}\left|(I-\Pi_N) \int_0^t e^{(t-s)A}dw(s)\right|_H^p \leq C T \left(\sum_{k=N+1}^\infty \lambda_{k}^{2}\alpha_k^{-\gamma} \right)^{p/2}.\]
  We can lower $p$ by using Jensen's inequality.
\end{proof}
\begin{proof}[Proof of Theorem \ref{thm:Galerkin}]
    First, we observe that in the case being considered
  \[|X^\e(t) - X^\e_N(t)|_H = |X^\e(t) - \Pi_N X^\e(t)|_H.  \]
  So, we can write
  \[X^\e(t) - \Pi_N X^\e(t) = (I-\Pi_N)e^{At}x + \sqrt{\e} (I-\Pi_N)\int_0^t e^{(t-s)A}B d w(s).\]
  We know that
  \[\left|(I-\Pi_N)e^{At}x \right|_H \leq |(I-\Pi_N)x|_H \to 0.\]
  If $x \in (-A)^{-\eta}(H)$, then
  \[|(I-\Pi_N)x|_H = |(I-\Pi_N)(-A)^{-\eta}(-A)^\eta x|_H \leq \alpha_{N+1}^{-\eta}|(-A)^\eta x|_H.\]
  The stochastic convolution term can be made small by Lemma \ref{lem:stoch-conv-Galerkin}.
  We can combine these estimates  to conclude that
  \[\E\sup_{t \leq T}|X^\e(t) - X_N^\e(t)|_H \leq  |(I-\Pi_N)x|_H + \sqrt{\e} CT\left(\sum_{k=N+1}^\infty \lambda_{k}^{2}\alpha_k^{-\gamma}  \right)^{1/2}. \]

  The above expression converges to $0$ and the convergence is uniform for initial conditions $x$ satisfying $|(-A)^{\eta}x|_H \leq R$.
\end{proof}

We conclude this section with two relevant remarks.
\begin{Remark}
The previous theorem shows that $X^\e(t)$ and its Galerkin approximation $X^\e_N(t)$ are pathwise close, but also that the Galerkin approximation's accuracy for fixed $n$ and $\e$ decreases as time $T$ increases. This is not a failure of our estimation. The difference $X^\e(t) - X^\e_N(t)$ is a Markov process that is exposed to the noise $\sqrt{\e}(I-\Pi_N)B d w(t)$. While this noise is degenerate in $H$, it is nondegenerate on the the subspace $(I-\Pi_N)(H)$. We can guarantee by standard arguments that for fixed $n$ and $\e$ and with probability one $X^\e(t)$ and $X^\e_N(t)$ will deviate from each other arbitrarily far on an infinite time horizon.
\end{Remark}

\begin{Remark}
  In Theorem \ref{thm:Galerkin}, we claimed that the convergence of the Galerkin approximations is uniform if the initial conditions $x$ are regular enough. In fact, over long time periods, the regularity of the initial conditions does not matter. This is because
  \[|(I-\Pi_N)e^{tA} x|_H \leq e^{-\alpha_N t} |x|_H.\]
  Therefore we can have uniform convergence on bounded sets in $D \subset H$ as long as we consider the estimate
  \[\sup_{x \in D}\E\sup_{t_0\leq t \leq T} \left|X^\e(t) - X^\e_N(t) \right|_H^2\]
  for some $t_0>0$.
\end{Remark}


\bibliographystyle{amsplain}

\end{document}